\documentclass[12pt]{amsart}
\usepackage{latexsym,amsmath,amssymb,amsfonts,amscd,graphics}
\usepackage[alphabetic,lite,nobysame]{amsrefs}
\usepackage{blindtext}
\usepackage[left=1in,right=1in,bottom=1.5in]{geometry}
\usepackage[skip=4pt plus1pt, indent=10pt]{parskip}
\usepackage[mathscr]{eucal}
\usepackage{mathrsfs}
\usepackage{tikz-cd}
\usepackage{hyperref, cleveref}
\usepackage{thmtools}
\usepackage{enumitem}
\usepackage[all]{xy}
\usepackage{caption}
\usepackage{tabularx} 
\usepackage{array}    
\setlist[itemize]{leftmargin=48pt}

\xyoption{rotate}
\hypersetup{
	colorlinks=true,
	linkcolor=blue,
	filecolor=blue,      
	urlcolor=cyan,
	citecolor=magenta}

\numberwithin{equation}{section}
\theoremstyle{plain}
\newtheorem{theorem}[equation]{Theorem}

\newtheorem{lemma}[equation]{Lemma}
\newtheorem{proposition}[equation]{Proposition}

\theoremstyle{remark}
\newtheorem{remark}[equation]{Remark}

\theoremstyle{definition}
\newtheorem{definition}[equation]{Definition}

\newtheorem{example}[equation]{Example}

\newcommand{\hideqed}{\renewcommand{\qed}{}}

\newcommand{\bP}{\mathbb{P}}

\newcommand{\bQ}{\mathbb{Q}}
\newcommand{\bZ}{\mathbb{Z}}

\newcommand{\bC}{\mathbb{C}}

\newcommand{\bL}{\mathbb{L}}

\newcommand{\calC}{\mathcal{C}}

\newcommand{\calM}{\mathcal{M}}
\newcommand{\calO}{\mathcal{O}}

\newcommand{\calP}{\mathcal{P}}

\newcommand{\calX}{\mathcal{X}}

\newcommand{\calD}{\mathcal{D}}

\newcommand{\Mov}{\mathrm{Mov}}

\newcommand{\Sat}{\mathrm{Sat}}
\newcommand{\SL}{\mathrm{SL}}

\newcommand{\Gr}{\mathrm{Gr}}

\newcommand{\rank}{\mathrm{rank}}

\newcommand{\NS}{\mathrm{NS}}
\newcommand{\Def}{\mathrm{Def}}
\newcommand{\Br}{\mathrm{Br}}

\newcommand{\Orth}{\mathrm{O}}

\newcommand{\Mon}{\mathrm{Mon}}

\newcommand{\sslash}{\mathbin{/\mkern-6mu/}}
\newcommand{\git}{/\kern-0.2em/}

\renewcommand{\div}{\mathrm{div}}
\newcommand{\pesk}{X_1^\sigma}
\newcommand{\DV}{X_6^\sigma}
\newcommand{\van}{\mathrm{van}}

\title[Peskine sixfolds \& Debarre--Voisin fourfolds with associated cubic fourfolds]{Peskine sixfolds and Debarre--Voisin fourfolds with associated cubic fourfolds}

\author{Corey Brooke}
\address{Department of Mathematics, Statistics, and Computer Science, St. Olaf College, Northfield, MN 55057, USA}
\email{brooke@stolaf.edu}

\author{Laure Flapan}
\address{Department of Mathematics\\
Michigan State University\\
619 Red Cedar Road, East Lansing, MI 48824}
\email{flapanla@msu.edu}

\author{Sarah Frei}
\address{Department of Mathematics MS 136, Rice University, 6100 S.~Main St, Houston, TX 77005, USA}
\email{sarah.frei@rice.edu}

\author{Lisa Marquand}
\thanks{Corresponding author: lisa.marquand@rutgers.edu}
\address{Department of Mathematics, Rutgers University - New Brunswick,
	Busch Campus,
	110 Frelinghuysen Road,
	Piscataway,
	NJ 8019, USA}
\email{lisa.marquand@rutgers.edu}

\begin{document}

\begin{abstract}
    We develop the notion of Peskine sixfolds with associated K3 surfaces and cubic fourfolds and work out numerical conditions for when these associations occur. In discriminant $24$, the first family for which there is an associated cubic fourfold, we identify the cubic explicitly. Moreover, we prove that in this case the Fano variety of lines of the cubic fourfold is isomorphic to the associated Debarre--Voisin hyperk\"ahler fourfold.
\end{abstract}

\maketitle

\section{Introduction}

There is a close connection between Fano varieties of K3 type and hyperk\"ahler manifolds of K3$^{[n]}$-type, most famously exhibited by a cubic fourfold and its Fano variety of lines \cite{BeauvilleDonagi}. 
More recently, there have been numerous similar constructions \cites{DV, DIM15, IlievManivelFanos, LLSvS, LSV, FatighentiMongardi, BernardaraFatighentiManivel}. 
Among these, the Debarre--Voisin hyperk\"ahler fourfold introduced in \cite{DV} remains somewhat mysterious. 
We recall the definition: given a $10$-dimensional complex vector space $V_{10}$ and a trivector $\sigma\in\bigwedge^3V_{10}^\vee$, the Debarre--Voisin manifold is defined by the degeneracy condition
\[
\DV=\{[V_6]\in\Gr(6,V_{10})\;|\;\sigma|_{V_6}=0\}.
\]
When $\sigma$ is general, this $\DV$ is a hyperk\"ahler fourfold. 
Initially, Debarre--Voisin observed that $\DV$ was  (rationally) Hodge-theoretically associated to a 20-dimensional hypersurface in $\Gr(3,V_{10})$. 
Later, an integral Hodge-theoretic association between $\DV$ and the so-called Peskine sixfold
\[
\pesk=\{[V_1]\in\bP(V_{10})\;|\;\rank(\sigma(V_1,-,-))\le6\},
\]
a Fano variety of K3 type, was described \cites{Han, BenedettiSong}. 
Here, we investigate divisors in the moduli space of trivectors $\sigma$ for which $\DV$ and $\pesk$ acquire Hodge-theoretic associations to K3 surfaces or cubic fourfolds; when such associations appear, it is natural to expect the birational geometry of $\DV$ and of $\pesk$ to be richer.

In the case of cubic fourfolds for instance, Hassett divisors $\mathcal{C}_d$ in the moduli space of cubic fourfolds parametrize those cubic fourfolds containing a surface not homologous to a complete intersection, where $d$ is the discriminant of the saturated lattice spanned by the class of the surface and the square of the hyperplane class.  
Hassett described the values of $d$ for which the cubic fourfolds $Y\in \mathcal{C}_d$ have associated K3 surfaces, meaning that the transcendental lattice $T_Y\subset H^4(Y,\mathbb{Z})$ is Hodge isometric to $T_S(-1)$ for some K3 surface $S$ \cite{Hassettcubicslong}.  
For such a  $Y\in \mathcal{C}_d$, the hyperk\"ahler fourfold given by the Fano variety of lines $F(Y)$ on $Y$ is always birational to a moduli space of stable sheaves on $S$ \cite{Add16}. 
Further, in this case, the cubic fourfold $Y$ is expected to be rational \cites{Hassettcubicslong, Kuznet, Add16, MR4292740}.

Here we study Peskine sixfolds $\pesk$ and their corresponding Debarre--Voisin hyperk\"ahler fourfolds $\DV$ in analogy with the case of cubic fourfolds $Y$ and their corresponding hyperk\"ahler fourfolds $F(Y)$. 
Towards this end, we study divisors $\mathcal{P}_d$ in the moduli space of trivectors $\sigma \in \bigwedge^3 V_{10}^\vee$ yielding marked Peskine sixfolds $\pesk$ or marked Debarre--Voisin fourfolds $\DV$ of discriminant $d$. 
We formalize the notion of such $\sigma\in \mathcal{P}_d$ having an associated K3 surface or an associated cubic fourfold and establish the following. 

\begin{theorem}\label{main thm 1}
    Let $\sigma\in \bigwedge^3 V_{10}^\vee$ be a very general trivector of discriminant d. Then:
    \begin{enumerate}
        \item $\sigma$ has an associated K3 surface if and only if $d$ satisfies the numerical conditions of Proposition \ref{prop:associated k3}, and
        \item $\sigma$ has an associated cubic fourfold $Y\in \calC_d$ if and only if $d$ satisfies the numerical conditions of Proposition \ref{prop:associated cubic}.
    \end{enumerate}
\end{theorem}

Along the way to proving Theorem~\ref{main thm 1}, we determine orbit representatives for the possible markings for smooth special Peskine varieties, see Proposition~\ref{prop:intersection forms}.

The first divisor $\calP_d$ parametrizing trivectors with an associated cubic fourfold occurs when $d=24$. 
When the trivector $\sigma $ is general in this divisor, the Peskine $\pesk$ has exactly one isolated singular point. More specifically, there is a unique flag $W_1\subset W_6\subset V_{10}$ with $\sigma(W_1,W_6,V_{10})=0$, and $\pesk$ is singular at $\bP(W_1)$ \cite{BenedettiSong}.
Further, Benedetti--Song show that such a Peskine variety $\pesk$ is such that $\sigma$ an associated twisted K3 surface $(S,\beta)$ with $\beta \in \Br (S)$. 
A main result of \cite{BenedettiFaenzi} is a rationality construction when $\pesk$ satisfies a cohomological condition, and Benedetti--Faenzi conjecture that this condition holds when $\beta$ vanishes. 
More recently, \cite{hassett24} studied the case of special cubic fourfolds of discriminant $d=24$, and showed that the obstruction to rationality is again governed by a twisted K3 surface.
Thus, it is natural to ask whether there is a geometric relationship between a Peskine $\pesk$ with $\sigma\in \calP_{24}$ and its associated cubic fourfold $Y\in \calC_{24}.$

\begin{theorem}\label{main thm 2}
Let $\sigma\in \bigwedge^3 V_{10}^\vee$ be a very general trivector of discriminant 24. Then:
\begin{enumerate}
    \item $Y:=\bP(W_6)\cap \pesk$ is a smooth cubic fourfold of discriminant $24$,
    \item the Fano variety of lines $F(Y)$ is isomorphic to the Debarre--Voisin variety $\DV$, and
    \item $Y$ is Hodge-theoretically associated to $\sigma$.
\end{enumerate}
\end{theorem}

\subsection*{Outline}
We recall the definitions of both the Peskine varieties and Debarre--Voisin varieties in Section \ref{sec:prelims}, along with relevant results on the period maps, moduli spaces and cohomology. 
In Section \ref{sec:special pesk}, we develop the notion of special trivectors and special Peskine varieties. 
In Section \ref{sec:assoc K3/cubic}, we introduce the notion of associated K3 surfaces and cubic fourfolds for special trivectors of discriminant $d$. 
We prove that when both are smooth, our definitions coincide with the Hodge-theoretic association analogous to the case of special cubic fourfolds \cite{Hassettcubicslong}. 
We prove Theorem \ref{main thm 1} in Propositions \ref{prop:associated k3} and \ref{prop:associated cubic}. 
In Section \ref{sec:discr 24}, we study special trivectors of discriminant 24 more closely, proving Theorem \ref{main thm 2}. 
Finally, in Appendix~\ref{appendix:examples} we provide an explicit example of a trivector with its associated Peskine variety to support claims made in Section~\ref{sec:discr 24}.

\subsection*{Acknowledgments}
We greatly thank Adam Logan, who provided code enabling us to construct Peskine varieties of discriminant 24 and verify the distinguished cubic fourfold is smooth (see Appendix~\ref{appendix:examples}).
We thank Vladimiro Benedetti for helpful discussions regarding the proof of \cite[Theorem~2.20]{BenedettiSong} and Alexander Kuznetsov for comments that helped us improve the clarity of Section~\ref{sec:discr 24}.
We appreciate productive conversations around the subject with Asher Auel and Yoon-Joo Kim. Finally, we thank the anonymous referees for their helpful comments and suggestions.

L.~Flapan was supported in part by NSF grant DMS-2200800. S.~Frei was supported in part by NSF grant DMS-2401601. L.~Marquand was supported by NSF grant DMS-2503390.
\subsection*{Disclosures}
There is no conflict of interest to disclose.
Data sharing not applicable to this article as no datasets were generated or analysed during the current study.

\section{Preliminaries}\label{sec:prelims}

Fixing a ten dimensional complex vector space $V_{10}$, a trivector $\sigma\in \bigwedge^3V_{10}^\vee$ defines the following varieties, given by degeneracy conditions:
\begin{itemize}
    \item \textbf{Peskine variety:} $\pesk:=\{[V_1]\in \bP(V_{10})\mid \rank(\sigma(V_1,-, -))\leq 6\}.$ The variety $\pesk$ has expected dimension 6 and is in general smooth.
    \item \textbf{Debarre--Voisin variety:} $\DV:=\{[V_6]\in \Gr(6, V_{10})\mid \sigma|_{V_6}=0\}.$ This has expected dimension 4, and for general $\sigma$ is a hyperk\"ahler manifold of K3$^{[2]}$-type \cite{DV}. 
\end{itemize}
The GIT quotient 
\[
\calM_\sigma:=\bP(\bigwedge^3V_{10}^\vee)\sslash\SL(V_{10})
\]
serves as a moduli space parametrizing both Peskine varieties and Debarre--Voisin varieties.

\subsection{Lattice theory}\label{section: lattice}
Let $L$ be an integral lattice. We denote either by $x\cdot y$ or by $x^2$ for the value of the bilinear form for $x,y\in L$. We say $L$ is even if $x^2\in 2\bZ$ for every $x\in L$. The \textit{divisibility} of $x\in L$ is the minimal generator of the ideal $x\cdot L\subset \bZ$; we denote this by $\div_L(x).$ The \textit{discriminant group} of $L$ is defined as $D(L):=L^*/L$ with induced quadratic discriminant form $q_L:D(L)\rightarrow \bQ/2\bZ$ if $L$ is even, and bilinear discriminant form $b_L:D(L)\times D(L)\rightarrow \bQ/\bZ$ if $L$ is odd. 
The lattice $L$ is \textit{unimodular} if the discriminant group is trivial.

A \textit{primitive embedding} of lattices \(M\hookrightarrow L\) is a injective linear map that preserving the quadratic forms such that \(L/M\) is a free abelian group. In this case we denote by \(M^{\perp}\) the \textit{orthogonal complement} of \(M\) in \(L\).
A primitive embedding \(M\hookrightarrow L\) is determined by a group \(H \subseteq D(L)\) and an isometry \(\gamma\colon H\to H'\subseteq D(M)\). Similarly, by \cite[Proposition 1.5.1]{nikulin}, a primitive embedding \(M\hookrightarrow L\) with \(M^\perp =T\) can be determined by a subgroup \(K\subseteq D(M)\) called the \textit{gluing subgroup} and an isometry \(\gamma: K\to K'\subseteq D(T(-1))\) called the \textit{gluing isometry}. In this case, the lattice $L$ is a finite index overlattice of $M\oplus T$. We refer to \cite[Proposition 1.5.1, Proposition 1.15.1]{nikulin} for more details on embedding of lattices. We will use the following lemma that follows from \cite[Proposition 1.15.1]{nikulin}:

\begin{lemma}\label{lem: discriminant}
    Let $L$ be an odd unimodular lattice, and $M\hookrightarrow L$ a primitive embedding with $b_M$ nondegenerate. Let $T:=(M)^\perp.$ Then we have:
    $$(D(T), b_T)\cong (D(M), -b_M),$$ where $b$ denotes the corresponding bilinear forms.
\end{lemma}
\begin{proof}
    This follows from \cite[discussion above 1.5.1 and \S 16]{nikulin}, but since it involves odd lattices we include the details.

    Consider $H:=L/(M\oplus T)\subset D(M)\oplus D(T)$. The group $H$ is isotropic for the bilinear form
    $$b_M\oplus b_T:(D(M)\oplus D(T))\times (D(M)\oplus D(T))\rightarrow \bQ/\bZ.$$ Further, primitivity implies the projections $H\rightarrow D(M), H\rightarrow D(T)$ are injective, and $|H|^2=|D(M)||D(T)|$ since $L$ is unimodular. Thus $H$ is the graph of an isomorphism $\gamma:D(M)\cong D(T)$, and bilinearity gives $b_T=-b_M.$
\end{proof}

We will also use the following theorem, recalled in \cite[Lemma 3.5]{GHS}. We denote by $\tilde{\Orth}(L)$ the subgroup of $\Orth(L)$ acting trivially on the discriminant group. 

\begin{theorem}[Eichler's Criterion]\label{theorem: eichler}
    Let $L$ be an even lattice containing two orthogonal isotropic planes, denoted $U\oplus U$. Then the $\tilde{\Orth}(L)$-orbit of a non-zero vector $v\in L$ is determined by $v^2$ and by the image $\left[\frac{v}{\div_L(v)}\right]\in D(L).$
\end{theorem}

\subsection{The period map}

Letting $\sigma$ be general so that the variety $\DV$ is indeed a hyperk\"ahler fourfold of K3$^{[2]}$-type, the lattice $(H^2(\DV,\bZ), q)$ is isomorphic to $\bL$, where
\[
\bL:=U^3\oplus E_8(-1)^2\oplus \langle -2\rangle.
\]
There is a polarization $H\in \NS(\DV)$ satisfying $q(H)=22$ and $q(H, x)\in 2\bZ$ for all $x\in H^2(\DV,\bZ)$ \cite[Theorem~1.2]{DV}. 
We denote by $\calM_{22}^{(2)}$ the coarse moduli space of K3$^{[2]}$-type manifolds with a polarization of square 22 and divisibility 2. 

By Eichler's criterion (Theorem \ref{theorem: eichler}), there is a unique orbit under $\Orth(\bL)$ of primitive elements $h\in \bL$ with $q_\bL(h)=22$ and $\div_\bL(h)=2$; we fix such an $h$. 
For any $\DV$, we fix a labeling $\phi:H^2(\DV,\bZ)\overset{\sim}\longrightarrow  \bL$ such that $\phi(H)=h$. For a family of hyperk\"ahler manifolds $\calX\rightarrow \Def(\DV)$, there exists (by Ehresmann's theorem)  a family of markings $\phi_t:H^2(\calX_t,\bZ)\rightarrow \bL$ with $\phi_0=\phi.$ This yields a period map
\[
\mathscr{P}: \calM_\sigma\dashrightarrow \calM_{22}^{(2)}\rightarrow \mathbb{D}/\Gamma, \hspace{1cm} \mathscr{P}([\sigma])= [\phi(H^{2,0}(\DV))],
\]
where 
\[
\mathbb{D}:=\{x\in \bP(\bL\otimes \bC)\mid x^2=x\cdot h=0, x\cdot \bar{x}>0\}
\]
and $\Gamma=\Orth(\bL, h)$.

Note that the map $\calM_\sigma\dashrightarrow \calM_{22}^{(2)}$ is birational \cite[Theorem 1.8]{OG22}. 
Further, the period map $\mathscr{P}$ is resolved by considering the Baily-Borel projective compactification $\overline{\mathbb{D}/\Gamma}$ of $\mathbb{D}/\Gamma$ to obtain a diagram:
\[
\xymatrix{ \widetilde{\calM}_\sigma \ar[r]^-{\widetilde{\mathscr{P}}} \ar[d]^{\varepsilon} & \overline{\mathbb{D}/\Gamma} \\
\calM_\sigma \ar@{-->}[r]_-{\mathscr{P}}  & \mathbb{D}/\Gamma, \ar@{^{(}->}[u] }
\]
where $\varepsilon$ is a birational morphism.
Let $v\in \langle h\rangle^\perp$ be a class of negative square. 
The Heegner divisor associated to $v$ is the image in $\mathbb{D}/\Gamma$ of the hypersurface:
\[
\{x\in \bP(\bL\otimes \bC) \mid x^2=x\cdot h=x\cdot v=0, x\cdot\bar{x}>0\}.
\]
Letting $N:=\langle h, v\rangle \subset \bL$, this divisor is equal to the image of $\bP(N^\perp\otimes \bC)\subset \bP(\bL\otimes \bC)$ in $\mathbb{D}/\Gamma$. 
The discriminant of $N^\perp\subset \bL$ is a positive even integer $d$; we denote this divisor by $\calD_{d}\subset \mathbb{D}/\Gamma$ and label the lattice $N_{d}:=N$.

The inverse image via the period map of the Heegner divisor $\calD_{d}$ inside $\calM_{22}^{(2)}$ is a Noether-Lefschetz divisor $\widetilde{\calD}_{d}\subset \mathcal{M}_{22}^{(2)}$ parametrizing smooth hyperk\"ahler fourfolds $X$ with Picard rank at least 2 and a sublattice $N_d\subset \NS(X)$. 
Note that for a general hyperk\"ahler fourfold $X\in \widetilde{\calD}_{d}$, we have $T(X)\cong N_{d}^\perp.$

Some divisors $\widetilde{\calD}_{d}$ are not in the image of the map $\calM_\sigma\dashrightarrow \calM_{22}^{(2)},$ meaning that the corresponding hyperk\"ahler fourfolds are not constructed from a trivector $\sigma$. These are the inverse images of the HLS divisors on $\overline{\mathbb{D}/\Gamma}$, introduced in \cite{DHOGV}. More precisely, an HLS divisor  is an irreducible hypersurface in $\overline{\mathbb{D}/\Gamma}$ which is the image by $\widetilde{\mathscr{P}}$ of an exceptional divisor of $\varepsilon$, so in particular has inverse image in $\calM_\sigma$ of codimension greater than 1. By combining \cite[Theorem~1.1]{DHOGV}, \cite[Corollary 1]{Oberdieck}, and \cite[Corollary~6.5]{SongspecialDV}, we get the following. 
    
\begin{proposition}
    \label{prop: HLS divisor}
    The Heegner divisor $\calD_d$ is an HLS divisor if and only if  $d\in \{2,6,8,10,18\}$. 
\end{proposition}

\subsection{The cohomology of $\pesk$}\label{sec: coh of pesk}

Let $\pesk$ be a smooth Peskine sixfold. 
We recall some basic facts about the cohomology of $\pesk$ from \cite[Section 2.6]{BenedettiSong}. 
The middle cohomology $H^6(\pesk, \bZ)$ equipped with the intersection pairing is a unimodular odd lattice of signature $(22,2)$, so it is isomorphic to $I_{22,2}:= \langle 1\rangle^{22}\oplus \langle -1\rangle^2$. 
We fix an isomorphism $H^6(X_1^\sigma,\bZ)\cong I_{22,2}$ throughout.

By \cite[Theorem 1.2]{BenedettiSong}, the integral Hodge conjecture holds for $\pesk$ in all degrees.
In particular, classes $\tau\in H^{3,3}(\pesk,\bZ)$ are represented by algebraic cycles. 
There are always two distinguished classes in the middle algebraic cohomology of $\pesk$: the cube $h^3$ of the hyperplane class and the class $\pi$ of a Palatini threefold, obtained as the linear section $\pesk\cap\bP (V_6)$ for any $[V_6]\in\DV$ \cite{Han}. 
Recall that a Palatini threefold is a smooth degree 7 threefold in $\bP^5$, which is a scroll over a smooth cubic surface. 
These classes span the sublattice
\[
\Lambda_{11}:=
\begin{pmatrix}
    15 & 7 \\
    7 & 4
\end{pmatrix}\subset I_{22,2},
\] see \cite[Proof of Proposition 2.15]{BenedettiSong}.
The vanishing cohomology of $\pesk$ is defined as
$$\Lambda_{\mathrm{van}}:=H^6(\pesk,\bZ)_{\mathrm{van}}:=\Lambda_{11}^\perp\subset H^6(\pesk,\bZ).$$ 
It is a polarized integral Hodge structure of K3 type with discriminant 11.

\begin{remark}\label{remark: why eichler}
    Note that as abstract lattices, $\Lambda_{\mathrm{van}}\cong \langle h\rangle^\perp\subset \bL$, where $h$ satisfies $q_\bL(h)=22$ and $\div_{\bL}(h)=2$ as before. In particular, $\Lambda_{\mathrm{van}}$ is even, and contains two copies of the lattice $U$, which we later use to apply Eichler's Criterion, Theorem~\ref{theorem: eichler}.
\end{remark}

An incidence correspondence referred to as the \emph{universal Palatini variety} relates the vanishing cohomology of $\pesk$ to the primitive cohomology of $\DV$:

\begin{theorem}\cite[Theorem 2.16]{BenedettiSong} \label{theo:hodge isometry}
For a trivector $\sigma\in \bigwedge^3 V_{10}^\vee$ such that both $\DV$ and $\pesk$ are smooth, there is an isomorphism of polarized integral Hodge structures
$$\phi:H^6(\pesk,\bZ)_{\mathrm{van}}\rightarrow H^2(\DV, \bZ)_{prim}(-1).$$    
\end{theorem}

\begin{remark}\label{v general DV}
    Note that for a very general trivector $\sigma,$ the Debarre-Voisin variety $\DV$ has Picard rank 1, with $T(\DV)\cong \Lambda_{\mathrm{van}}(-1)$, or equivalently $T(\pesk) \cong \Lambda_\van$.
\end{remark}

\section{Special Peskine sixfolds}\label{sec:special pesk}

We continue to fix a 10-dimensional complex vector space $V_{10}$, and let $\sigma\in \bigwedge^3V_{10}^\vee$ be a trivector for which $\pesk$ has the expected dimension. 
In this section, we define marked and special Peskine sixfolds $\pesk$ and trivectors $\sigma$. 
We then classify the possible markings for special Peskine sixfolds.

\begin{definition}
    Let $\pesk$ be a smooth Peskine sixfold.
    We say that $\pesk$ is a \textbf{marked} Peskine sixfold of \textbf{discriminant} $d$ if there exists a saturated rank-3 sublattice $M_d$ of discriminant $d$ with
    \[
    \Lambda_{11}\subset M_d\subset H^{3,3}(X_1^\sigma, \bC)\cap H^6(\pesk,\bZ),
    \]
    and we call $M_d$ a \textbf{marking} of $X_1^\sigma$.
\end{definition}

In particular, a very general marked Peskine sixfold $\pesk$ of discriminant $d$  has a unique marking, given by  $M_d= H^{3,3}(\pesk,\bC)\cap H^6(\pesk,\bZ)$. 
Then since $H^6(X_1^\sigma,\bZ)$ is unimodular, both of the lattices $N_d^\perp$ and $M_d^\perp$ have discriminant $d$ and so by the Hodge isometry in Theorem~\ref{theo:hodge isometry}, 
\[
H^6(\pesk,\bZ)\supset M_d^\perp\cong T(\pesk)\cong T(\DV)(-1)\cong N_{d}^\perp (-1)\subset H^2(\DV,\bZ)(-1),
\]
so the image of $\sigma$ under the period map lies in the Heegner divisor $\calD_d\subset \mathbb{D}/\Gamma$. 
It follows that the preimage of the Heegner divisor $\calD_{d}$ under the period map, which we denote by
\[
\calP_d:=\mathscr{P}^{-1}(\calD_d)\subset \calM_\sigma,
\]
has general points parametrizing marked Peskine sixfolds of discriminant $d$.

\begin{remark}
    Whenever $\calD_d$ is not a HLS-divisor, $\calP_d$ is a non-empty divisor (see Proposition \ref{prop: HLS divisor}). In particular, all $\calP_d$ are divisors for $22 \leq d$.
\end{remark}

In \cite[Proposition 2.9]{BenedettiSong}, the authors identify two divisors $D^{1,6,10}, D^{3,3,10}\subset \calM_\sigma$ that together parametrize the singular Peskine sixfolds. More precisely, these divisors have the following definition:

\begin{definition}\label{defn:flag divisor}
    A trivector $\sigma\in D^{1,6,10}$ if there exists a flag $V_1\subset V_6\subset V_{10}$ such that $\sigma(V_1,V_6,V_{10})=0$. Similarly,  $\sigma\in D^{3,3,10}$ if there exists a $V_3\subset V_{10}$ such that $\sigma(V_3, V_3,V_{10})=0.$
\end{definition}
 
The period map $\mathscr{P}$ maps $D^{1,6,10}$ birationally onto $\calD_{24}\subset \mathbb{D}/\Gamma$ and $D^{3,3,10}$ birationally onto $\calD_{22}\subset \mathbb{D}/\Gamma$. We thus set
\begin{align*}
\calP_{22}&:=\calD^{3,3,10}\simeq \mathscr{P}^{-1}(\calD_{22}), \\
\calP_{24}&:=\calD^{1,6,10}\simeq \mathscr{P}^{-1}(\calD_{24}).
\end{align*}
Note that for a general $\sigma\in \calP_{24},$ the Peskine sixfold $\pesk$ has an isolated singularity at $[V_1]\in \pesk.$ A general $\sigma\in \calP_{22}$ has $\pesk$ singular along the plane $\bP([V_3])$.

This allows us to consider singular Peskine sixfolds when making the following definition:
\begin{definition}\label{defn:specialtrivector}
    A \textbf{special Peskine sixfold of discriminant $d$} is a Peskine sixfold $\pesk$ such that $\sigma\in \calP_d.$
    In this case, we also say that the trivector $\sigma\in \bigwedge^3V_{10}^\vee$ is \textbf{special of discriminant }$d$.
\end{definition}

For $\sigma\in \calP_d\setminus(\calP_{22}\cup \calP_{24})$, the special Peskine sixfold $\pesk$ admits a marking of discriminant $d$. 
However, for $\sigma\in \calP_{22}\cup \calP_{24},$ the Peskine sixfold $\pesk$ is singular. 
In these two cases, it is not clear whether $\pesk$ admits a marking since $H^6(\pesk,\bZ)$ may not be a pure Hodge structure.

\begin{remark}
    Since the Debarre--Voisin variety associated to a general trivector $\sigma\in \calP_{24}$ is smooth \cite[Lemma 2.1]{BenedettiSong} and $\pesk$ has only isolated singularities \cite[Proposition 2.9]{BenedettiSong}, we wonder whether the incidence correspondence of \cite[Theorem 2.16]{BenedettiSong} can be extended to prove that the Hodge structure on $H^6(\pesk,\bZ)_{\mathrm{van}}$ remains pure. 
    Alternatively, one could hope that the singularity of $\pesk$ is suitably mild enough (as in the case of cubic fourfolds with $ADE$-singularities) to ensure the purity.
    Computer algebra systems such as Magma and Macaulay2 were unable to compute invariants of the singularity.
\end{remark}

To complete our terminology, we make the following definition:

\begin{definition}
    A \textbf{marked trivector} $\sigma$ \textbf{of discriminant} $d$ is a trivector $\sigma\in \calP_d\setminus \calP_{22}$ and an  embedding $H\in N_d\subset \NS(\DV),$ where $N_d$ is a saturated rank 2 lattice with $N_d^\perp\subset H^2(\DV,\bZ)$ of discriminant $d$.
\end{definition}

Note that we omit $\calP_{22}$ from the definition because in that case, both $\pesk$ and $\DV$ are singular. 
When $\sigma\in \calP_d\setminus(\calP_{22}\cup \calP_{24})$, the Debarre--Voisin fourfold $\DV$ and the Peskine sixfold  $\pesk$ both have markings with the same discriminant $d$ and the same discriminant form.

We recall restrictions on the discriminants of Peskine sixfolds:

\begin{proposition}\cite[Proposition 4.1]{DM19} \label{prop:d mod 22}
    If $X_1^\sigma$ is a smooth special Peskine of discriminant $d$, then $d>0$, and $d\equiv0,2,6,8,10$ or $18\bmod22$.
\end{proposition}

To build on this result, we classify the markings of special Peskine sixfolds, similar to the analogous result for Gushel-Mukai fourfolds from \cite[Proposition 6.2]{DIM15}.  
Let 
\[
\widetilde{\Orth}(\Lambda_{\mathrm{van}})=\{g\in\Orth(I_{22,2})\;|\;g|_{\Lambda_{11}}=\mathrm{id}\},
\]
and for each $d$, let $\calO_d$ be the set of orbits for the action of $\widetilde{\Orth}(\Lambda_{\mathrm{van}})$ on the set of primitive, positive definite, rank 3 sublattices $M_d\subset I_{22,2}$ of discriminant $d$ containing $\Lambda_{11}$. 

\begin{remark}
    Note that $\Gamma=\widetilde{\Orth}(\Lambda_\van)$; indeed, $\Gamma$ is clearly a subgroup of $\Orth(\Lambda_\van)$, consisting of isometries acting trivially on $D(\Lambda_{\van})$. 
    For $g\in \widetilde{\Orth}(\Lambda_\van)$, $g$ acts trivially on $D(\Lambda_{11})$.
    Since $I_{22,2}$ is unimodular, $g$ preserves the gluing group and thus acts trivially on $D(\Lambda_{\van})$.
\end{remark}

\begin{proposition}\label{prop:intersection forms}
Let $\pesk$ be a smooth special Peskine sixfold of discriminant $d$. For each integer $d>0$ with $d\equiv0,2,6,8,10$, or $18\bmod22$, the set of orbits $\calO_d$ consists of a single element, which is represented by
\[
M_d\simeq\begin{pmatrix}
            15 & 7 & a \\
            7 & 4 & b\\
            a & b & c
        \end{pmatrix},
\]
where
    \begin{enumerate}
        \item[(a)] if $d\equiv0\bmod22$, then $(a,b,c)=(0,0,d/11)$;
        \item[(b)] if $d\equiv2\bmod22$, then $(a,b,c)=(3,1,(d+9)/11)$;
        \item[(c)] if $d\equiv6\bmod22$, then $(a,b,c)=(1,1,(d+5)/11)$;
        \item[(d)] if $d\equiv8\bmod22$, then $(a,b,c)=(2,1,(d+3)/11)$;
        \item[(e)] if $d\equiv10\bmod22$, then $(a,b,c)=(3,2,(d+12)/11)$;
        \item[(f)] if $d\equiv18\bmod22$, then  $(a,b,c)=(1,0,(d+4)/11)$.
    \end{enumerate}
\end{proposition}

\begin{proof}
    Suppose that $\pesk$ is a smooth Peskine sixfold of discriminant $d$ with a marking $M_d$. 
    Let $w$ be a generator of the rank $1$ sublattice $M_{\mathrm{van}}:=M_d\cap \Lambda_{\mathrm{van}}$.
    By Remark \ref{remark: why eichler}, $\Lambda_{\mathrm{van}}$ contains two copies of $U$ and hence satisfies Eichler's criterion (see Theorem~\ref{theorem: eichler}). It follows that the orbit of a primitive vector $w\in\Lambda_{\mathrm{van}}$ under the action of $\widetilde{\Orth}(\Lambda_{\mathrm{van}})$ is determined by $w^2$ and the image $w_*=w/\div_{\Lambda_{\mathrm{van}}} w$ of $w$ in the discriminant group $D(\Lambda_{\mathrm{van}}):= \Lambda_{\mathrm{van}}^\vee/\Lambda_{\mathrm{van}}$. Recall that $D(\Lambda_{\mathrm{van}})\cong \bZ/11\bZ,$ and hence $\div_{\Lambda_{\mathrm{van}}} w= 1$ or 11 (see also Remark \ref{v general DV}).

    First, we consider the case $\div_{\Lambda_{\mathrm{van}}}(w)=1$, so that the class $w_*\in D(\Lambda_{\mathrm{van}})$ is trivial. 
    Let $w'\in M_d$.  
    We can choose coprime integers $a$, $b$, and $c$ so that (in the notation of Section \ref{sec: coh of pesk})
    \[
    ah^3+b\pi+cw=mw'
    \]
    for some integer $m$. 
    Since $\div_{\Lambda_{\mathrm{van}}}(w)=1$, there is some class $w''\in \Lambda_{\mathrm{van}}$ for which $w\cdot w''=1$, and pairing the equation above with $w''$ yields $c=m(w'\cdot w'')$. 
    So then
    \[
    ah^3+b\pi=m(w'-(w'\cdot w'')w).
    \]
    In particular, $m$ divides $a$, $b$, and $c$. 
    Since $a, b$ and $c$ are coprime, we must have $m=1$ and $w'\in \Lambda_{11}\oplus\bZ w$. 
    It follows that $M_d=\Lambda_{11}\oplus\bZ w$, and $d\equiv 0 \bmod 11$. 
    
    To complete the proof of (a), it remains to find a class $w\in\Lambda_{\mathrm{van}}$ with $\div_{\Lambda_{\mathrm{van}}}(w)=1$ and $w^2=2k$ for each integer $k$. 
    For this, choose a hyperbolic factor $U$ of $\Lambda_{\mathrm{van}}$ with standard generators $u$ and $v$, and take $w=u+kv$. Note that since $U$ is unimodular and saturated in $\Lambda_{\van}$, the sublattice $\Lambda_{11}\oplus U$ is saturated in $I_{22,2}$. Since $w$ is primitive in $U$, it follows that $M_d\subset I_{22,2}$ is saturated.

    Now, consider the case $\div_{\Lambda_{\mathrm{van}}}(w)=11$. 
    Then $0\neq w_*\in D(\Lambda_{\mathrm{van}})$, so $w_*$ generates $D(\Lambda_{\mathrm{van}})\cong \bZ/11\bZ$, and since the discriminant of $M_d$ is $d$, the index of $\Lambda_{11}\oplus\bZ w$ in $M_d$ must be $11$. 
    In particular, $w^2=11d$ where $d$ is the discriminant of $M_d$. 
    Since $w$ is only determined up to multiplication by $\pm1$, we consider five of the ten generators of $D(\Lambda_{\mathrm{van}})$ separately.

    Recall that $D(\Lambda_{11})\cong \bZ/11\bZ$, and inspecting the Gram matrix one sees that $\div_{\Lambda_{11}}(h^3+\pi)=11$, so  $\left[\frac{h^3+\pi}{11}\right]$ is a generator. If under the gluing isometry $D(\Lambda_{\mathrm{van}})\overset{\sim}\to D(\Lambda_{11})$ (see Section \ref{section: lattice}), $w_*\mapsto[\frac{h^3+\pi}{11}]$, then $\tau=\frac{h^3+\pi+w}{11}\in I_{22,2}$. 
    From the fact that 
    \[
    \tau^2=\frac{33+w^2}{11^2}=\frac{3+d}{11}
    \]
    is an integer, we find $d\equiv8\bmod11$. 
    Moreover, $\tau\in M_d$, and $\Lambda_{11}\oplus\bZ w$ has index $11$ in $M_d$. 
    Thus $M_d$ is generated by $h^3$, $\pi$, and $\tau$, and calculating the intersection form on $M_d$ yields the representative in (d). 

    More generally, if $w_*\mapsto \left[\frac{k(h^3+\pi)}{11}\right]$ for $1\leq k \leq 5$, then $\tau=\frac{k(h^3+\pi)+w}{11}\in I_{22,2}.$ 
    Thus $\tau^2=\frac{3k^2+d}{11}$ is an integer, and $d\equiv -3k^2 \bmod 11.$ 
    The result then follows by varying $k$ and computing the intersection form under the basis $h^3$, $\pi$, and $\pi-\tau$ in the cases $k=2,3$ (corresponding to (e) and (c), respectively) and the basis $h^3$, $\pi$, and $\tau-\pi$ in the cases $k=4,5$ (corresponding to (f) and (b), respectively).
\end{proof}

For later use, we record the value of the bilinear form $b_{M_d}(-,-)\in \bQ/\bZ$ on $D(M_d):=M_d^*/M_d$. This result is analogous to \cite[Proposition 3.2.5]{Hassettcubicslong}.
\begin{lemma}\label{lemma:disc form}
    Let $X_1^\sigma$ be a smooth special Peskine sixfold of discriminant $d$ with a marking $M_d$. 
    Let $\{h^3,\pi,\tau\}$ be a basis for $M_d$ so that the intersection form on $M_d$ is as in Proposition~\ref{prop:intersection forms}.
    \begin{itemize}
        \item[(a)] If $d\equiv0\bmod22$, then $D(M_d)\cong\bZ/11\bZ\times \bZ/d'\bZ$ where $d'=d/11$. Moreover, if $11\nmid d'$, then $D(M_d)$ is cyclic, and the discriminant bilinear form takes value
        \[
        b_{M_d}\left(\frac{h^3+\pi}{11}+\frac{\tau}{d'}\right)=\frac3{11}+\frac1{d'}\in\bQ/\bZ
        \]
        on a generator of $D(M_d)$.
        \item[(b)] If $d\equiv 2,6,8,10,$ or $18\bmod22$, then $D(M_d)\cong\bZ/d\bZ$, and the discriminant bilinear form takes value $\frac{11}d$ on a generator of $D(M_d)$.
    \end{itemize}
\end{lemma}
\begin{proof}
Case $(a)$ follows directly from the fact that $M_d=\Lambda_{11}\oplus \langle \tau\rangle$ relabelling $w$ as $\tau$. 
Similarly, in case $(b)$ we have that $\Lambda_{11}\oplus \langle w\rangle \subset M_d$ with index 11. 
It follows from Nikulin's theory of overlattices \cite[Proposition 1.4.1]{nikulin} that $D(M_d)=\bZ/d\bZ$, and the value of the discriminant form follows directly.
\end{proof}

 Since the lattice $H^6(X_1^\sigma,\bZ)$ is unimodular, if $D(M_d)$ is cyclic then we can apply Lemma \ref{lem: discriminant} to conclude $D(M_d^\perp)$ is also cyclic with $b_{M_d^\perp} = -b_{M_d}.$ We will need the value of the quadratic form $q_{M_d^\perp}$ on a generator of $D(M_d^\perp)$.

\begin{lemma}\label{lem: disc1}
    If $d\equiv 0\bmod 22$ with $11^2\nmid d$, then the quadratic form $q_{M_d^\perp}$ takes value $\frac{8}{11}-\frac{11}{d}$ on a generator of $D(M_d^\perp).$
\end{lemma}
\begin{proof}
    In this case, $M_d=\Lambda_{11}\oplus \langle w\rangle,$ where $w^2=\frac{d}{11}$. Thus $M_d^\perp = \langle w\rangle^\perp\subset \Lambda_{\van},$ and so $$(D(M_d^\perp), q_{M_d^\perp})=(D(\Lambda_{\van}), q_{\Lambda_{\van}})\oplus \left\langle -\frac{11}{d}\right\rangle.$$
    It remains to show that $q_{\Lambda_{\van}}$ takes value $\frac{8}{11}\bmod \bQ/2\bZ$ on a generator. Indeed, $b_{\Lambda_{\van}}=-b_{\Lambda_{11}}$, and we have already seen that $b_{\Lambda_{11}}(\frac{h^3+\pi}{11})=\frac{3}{11}.$ Thus there is a generator $\alpha\in D(\Lambda_{\van})$ with either $q_{\Lambda_{\van}}(\alpha)=-\frac{3}{11}$ or $-\frac{3}{11}+1$ modulo $2\bZ$. Since $D(\Lambda_{\van})\cong\bZ/11\bZ$, we have 
    \[
    11^2q_{\Lambda_{\van}}(\alpha)=q_{\Lambda_{\van}}(11\alpha)=q_{\Lambda_{\van}}(0)=0\bmod2\bZ,
    \]
    which rules out the first case.
\end{proof}

\begin{lemma}\label{lem: discr2}
    If $d\equiv 2,6,8,10$ or $18\bmod 22$, then the quadratic form $q_{M_d^\perp}$ takes value $-\frac{11}{d}$ on a generator of $D(M_d^\perp).$
\end{lemma}
\begin{proof}
    Since the lattice $H^6(X_1^\sigma,\bZ)$ is unimodular, we can apply Lemma \ref{lem: discriminant} to conclude $D(M_d^\perp)$ is also cyclic with $b_{M_d^\perp} = -b_{M_d}.$ Thus the discriminant bilinear form of $D(M_d^\perp)$ takes value $-11/d$ on some generator $x\in D(M_d^\perp)$. It follows that $q_{M_d^\perp}(x)=-\frac{11}{d}+\epsilon \in \bQ/2\bZ$, with $\epsilon\in \{0,1\}$. Call these two quadratic forms $q_0, q_1$ for each value of $\epsilon.$ 
    By \cite[Theorem 1.3.3]{nikulin}, two finite quadratic forms $q_1, q_2$ with isomorphic bilinear forms are themselves isomorphic if and only if $\mathrm{sign}(q_1)\equiv \mathrm{sign}(q_2) \bmod 8$.  We can compute that $\mathrm{sign}(q_{M_d^\perp})\equiv 19-2\equiv 1\bmod 8.$ We will conclude the lemma by constructing an auxiliary even lattice $K$ with quadratic form $q_K\cong q_{M_d^\perp}$ taking value $-\frac{11}{d}$ on a generator. 

    Note that $\frac{d}{2}\equiv 1, 3, 4, 5, 9\bmod 11$, i.e. a nonzero quadratic residue modulo 11. Choose $r\in \bZ$ such that 
    $$r^2\equiv \frac{d}{2}\bmod{11}$$ and set $c:=\frac{2r^2-d}{22}\in \bZ.$ Let $K$ be the even lattice with Gram matrix:
    $$G_K:=\begin{pmatrix}
        2&1&0\\
        1& 6&r\\
        0& r& 2c
    \end{pmatrix}.$$ The determinant of $K$ is $-d$, and $K$ is even of signature $(2,1)$. Let $e_1, e_2,e_3$ be a basis for $K$. One can compute that $D(K)\cong \bZ/d\bZ,$ with generator $[e_3^\vee]:=-\frac{1}{d}(re_1-2re_2+11e_3)$. In particular, $q_K([e_3^\vee])=-\frac{11}{d}\bmod 2\bZ.$ Finally, since $b_K\equiv q_K\bmod \bZ$ and $\mathrm{sign}(q_K)=2-1=1 \bmod 8,$ we see that $q_K\cong q_{M_d^\perp}.$
\end{proof}

\section{Associated varieties: K3 surfaces and cubic fourfolds}\label{sec:assoc K3/cubic}

In this section, we formalize the notion of a special trivector $\sigma$ having an associated K3 surface (Section~\ref{subsec:K3}) or an associated cubic fourfold (Section~\ref{subsec:cubicfourfold}). 
In both cases, we make explicit  necessary and sufficient conditions for the existence of such an associated variety.

\subsection{Associated K3 surfaces}\label{subsec:K3}
Let $\sigma\in \calP_{d}\setminus(\calP_{22}\cup \calP_{24})$ be a trivector marked by $N_d\subset\NS(\DV)$. 
Since in this case the Peskine sixfold $\pesk$ is smooth, it admits a marking $M_d\subset H^6(\pesk,\bZ)$. 
It sometimes happens that the transcendental cohomology of $\pesk$ matches that of a K3 surface, i.e. there exists a polarized K3 surface $(S,f)$ and a Hodge isometry 
\[
H^6(X_1^\sigma,\bZ)\supset M_d^\perp\overset{\sim}\longrightarrow f^\perp\subset H^2(S,\bZ)(-1).
\]
In this case, it is natural to regard $S$ as associated to $\pesk$. 

On the other hand, if $\sigma \in \calP_{24}\setminus \calP_{22}$, then $\pesk$ has isolated singularities and may not admit a marking. 
In that case, the Debarre--Voisin fourfold $\DV$ is still smooth, and we can instead use the marking $N_d$ on $\DV$ to identify an associated K3 surface. 
For that reason, we give the following definition of associated K3 surfaces:

\begin{definition}\label{defn: K3 associated}
    Let $\sigma\in \calP_{d}\setminus \calP_{22}$ for some $d$, and let $N_d\subset \NS(\DV)$ be a corresponding marking of $\DV$. 
    Then a polarized K3 surface $(S,f)$ of degree $d$ is \textbf{associated to $\sigma$} if there exists a Hodge isometry
    \[
    H^2(\DV,\bZ)\supset N_d^\perp\overset{\sim}\longrightarrow f^\perp\subset H^2(S,\bZ).
    \]
\end{definition}
In the case of $\sigma\in \calP_{22}$ the Debarre--Voisin manifold $\DV$ is singular but admits a symplectic resolution $\widetilde{\DV}$. Thus the  transcendental cohomology $T(\DV)$ will be isomorphic to that of the symplectic resolution. We say that $\sigma$ has an associated K3 surface if the hyperk\"ahler manifold $\widetilde{\DV}$ does as above. In fact, such a $\sigma$ does have an associated K3: it was shown in \cite{DV} that $\widetilde{\DV}$ is birational to $S^{[2]}$ for an associated K3.

The following lemma shows that both notions of an associated K3 surface (either via the cohomology of $\pesk$ or that of $\DV$) coincide when both varieties are smooth.

\begin{lemma}\label{lem: equiv defn assoc k3}
    Let $\sigma\in \calP_d\setminus (\calP_{22}\cup \calP_{24})$.  A K3 surface $(S,f)$ of degree $d$ is associated to $\sigma$ if and only if there is a marking $M_d\subset H^{3,3}(\pesk,\bC)\cap H^6(\pesk,\bZ)$ and a Hodge isometry
    \[
    H^6(\pesk,\bZ)\supset M_d^\perp\longrightarrow f^\perp(-1)\subset H^2(S,\bZ)(-1).
    \]
\end{lemma}

\begin{proof}
    Note that for such a trivector $\sigma$, both $\pesk$ and $\DV$ are smooth. 
    Since $(S,f)$ is associated to $\sigma,$ there exists some marking $N_d\subset \NS(\DV)$ with $N_d^\perp\cong f^\perp$. 
    Note that $d$ is the discriminant of $N_d^\perp$. 
    Let $v$ be the generator of $N_d\cap \langle h\rangle^\perp\subset H^2(\DV,\bZ)_{prim}.$ 
    Then $v$ maps isometrically to some algebraic class $\tau\in H^6(\pesk,\bZ)_{\mathrm{van}}$ by Theorem \ref{theo:hodge isometry}.
    We have the following Hodge isometries: 
    \begin{equation}\label{isom}
    (\tau)^\perp_{H^6(\pesk,\bZ)_{\mathrm{van}}}\cong (v)^\perp_{H^2(\DV,\bZ)_{prim}}(-1)= N_d^\perp(-1) \subset H^2(\DV,\bZ)(-1).
    \end{equation}
    
    Let $M:=\Sat(\Lambda_{11}, \tau)\subset H^6(\pesk,\bZ).$ 
    Since $H^6(\pesk,\bZ)\supset M^\perp= \tau^\perp\subset H^6(\pesk,\bZ)_{\mathrm{van}}$ and $H^6(\pesk,\bZ)$ is unimodular, the discriminant of $M$ must be $d$. 
    Hence $M_d:=M$ gives a labelling of discriminant $d$, and by construction the desired Hodge isometry holds.

    Conversely, suppose  $M_d\subset H^6(\pesk,\bZ)$ is a marking such that 
    \[
    M_d^\perp\cong f^\perp(-1)\subset H^2(S,\bZ)(-1)
    \]
    for some K3 surface $(S,f)$ of degree $d$. Then $M_d\cap H^6(\pesk,\bZ)_{\mathrm{van}}$ is generated by an algebraic class $\tau,$ mapping isometrically to some algebraic class $v\in H^2(\DV,\bZ)_{prim}$ again via Theorem \ref{theo:hodge isometry}. One can run the Hodge isometries of Equation \ref{isom} backwards, to obtain a marking $N_d$ of $\DV$ and the desired Hodge isometry.
\end{proof}

We classify which special Peskine sixfolds arise from trivectors with associated K3 surfaces:

\begin{proposition}\label{prop:associated k3}
    Suppose $X_1^\sigma$ is a special Peskine sixfold of discriminant $d$.
    Then $\sigma$ has an associated K3 surface $(S,f)$ of degree $d$ if and only if $d$ is not divisible by $4$ or $121$, and every odd prime dividing $d$ is a square modulo $11$.
\end{proposition}
\begin{proof}
Recall that in the case $d=22$, for $\sigma$ to have an associated K3 surface means that the symplectic resolution $\widetilde{\DV}$ of $\DV$ has an associated K3 surface: this is the case, as first shown in \cite{DV}. The case of $d=24$ was studied in \cite[Proposition~2.4, Theorem~4.13]{BenedettiSong}, so we assume $d>24$.
In particular, $\pesk$ is smooth, and there exists a marking $\Lambda_{11}\subset M_d\subset H^{3,3}(\pesk)\cap H^6(\pesk,\bZ).$
By the Torelli theorem, $\sigma$ has an associated polarized K3 surface if and only if there is some class $f$ in the K3 lattice $\Lambda_{\mathrm{K3}}(-1):=U^3\oplus E_8^2$ such that $M_d^\perp\cong f^\perp$. Note for $f\in \Lambda_{K3}(-1)$, the lattice $f^\perp$ is unique in its genus. 
Such an isometry exists if and only if there is an isomorphism $D(M_d^\perp)\to D(f^\perp)$ respecting the discriminant form on each group. 

By Eichler's Criterion (Theorem~\ref{theorem: eichler}), we can assume $f=u+\frac{d}2v$ where $u$ and $v$ are standard generators for a hyperbolic factor of $\Lambda_{\mathrm{K3}}$, we see $D(f^\perp)\cong\bZ/d\bZ$. 
Moreover, for any generator $w$ of $D(f^\perp)$, the discriminant form $q$ takes value $q(w)=\frac{k^2}{d}\in\bQ/2\bZ$ for some integer $k$.

First, suppose $d\not\equiv0\bmod22$. 
By Lemma~\ref{lemma:disc form} and Lemma \ref{lem: discr2}, $D(M_d^\perp)\cong\bZ/d\bZ$, and the discriminant form takes value $-\frac{11}d$ on one of the generators. 
Thus, for $\sigma$ to have an associated K3 surface, we would need $-\frac{11}d=\frac{k^2}d\in\bQ/2\bZ$, i.e. for $-11$ to be a square modulo $2d$. 
Writing $d=2^\ell p_1^{e_1}\dots p_m^{e_m}$ for distinct odd primes $p_i\neq11$, this is equivalent to saying $-11$ is a square modulo $2^{\ell+1}$ and modulo $p_i$ for each $i$.

If $\ell\ge2$, then reducing modulo $8$ would yield $5=k^2\bmod8$, a contradiction. 
Hence $\ell=1$. By quadratic reciprocity,
\[
\left(\dfrac{-11}{p_i}\right)=\left(\dfrac{-1}{p_i}\right)\left(\dfrac{11}{p_i}\right)=(-1)^{\frac{p_i-1}{2}}(-1)^{\frac{p_i-1}{2}}\left(\dfrac{p_i}{11}\right),
\]
so $-11$ is a square modulo $p_i$ if and only if $p_i$ is a square modulo $11$. This completes the first case.

Now, suppose $d\equiv0\bmod22$. Since $D(f^\perp)$ is cyclic, Lemma~\ref{lemma:disc form} shows that $\sigma$ can only have an associated K3 surface if $11^2\nmid d$. 
Writing $d=11d'$ and arguing as before but now with Lemma~\ref{lem: discr2}, $\sigma$ has an associated K3 surface if and only if there is some integer $k$ for which $\frac8{11}-\frac1{d'}=\frac{k^2}{d}\in\bQ/2\bZ$. 
In other words, $\sigma$ has an associated K3 surface if and only if
\[
8d'-11\equiv k^2\bmod2d.
\]
Again, write $d=2^\ell\cdot11p_i^{e_1}\dots p_m^{e_m}$ for distinct odd primes $p_i\neq11$. 
We must show $8d'-11$ is a square modulo $2^{\ell+1}$, $11$, and each $p_i$. 
Reducing modulo $p_i$, we see as before that $8d'-11$ is a square modulo $p_i$ if and only if $p_i$ is a square modulo $11$. 
In that case, reducing modulo $11$ yields $8d'\equiv k^2\bmod 11$, so
\[
1=\left(\dfrac{8d'}{11}\right)=\left(\dfrac{2}{11}\right)^{\ell+3}\left(\dfrac{p_1}{11}\right)^{e_1}\dots\left(\dfrac{p_m}{11}\right)^{e_m}=(-1)^{\ell+1},
\]
so $\ell$ is odd. 
Reducing modulo $8$ shows that there is no solution if $\ell\ge3$, so there can only be a solution if $\ell=1$. There is a solution modulo $4$, completing the argument for this case.
\end{proof}

The first few examples of integers $d$ for which the trivector $\sigma$ of a general Peskine sixfold of discriminant $d$ has an associated K3 surface are $d=30$, $46$, $50$, $54$, $62$, $66$, $74$, $90$, and $94$.

\subsection{Associated cubic fourfolds}\label{subsec:cubicfourfold}

The transcendental cohomology of a Peskine sixfold can also be embedded into the primitive cohomology of a cubic fourfold, so we also develop the notion of associated cubic fourfolds. 
For a smooth cubic fourfold $Y$, one defines the lattice of algebraic cycles
\[
A(Y)=H^4(Y,\bZ)\cap H^{2,2}(Y,\bC),
\]
and the primitive cohomology 
\[
H^4(Y,\bZ)_{prim}=\langle\eta_Y\rangle^\perp\subset H^4(Y,\bZ),
\]
where $\eta_Y$ is the square of the hyperplane class. Let
\[
A(Y)_{prim}=H^4(Y,\bZ)_{prim}\cap A(Y).
\]
There are countably many divisors $\calC_d$ in the moduli space of cubic fourfolds parametrizing those with $A(Y)_{prim}\neq0$ \cite{Hassettcubicslong}.
Whenever $\eta_Y\in K_d\subset A(Y)$ where $K_d$ is a primitive sublattice of rank $2$ and discriminant $d$, one says $Y\in\calC_d$, and $K_d$ is called a marking of $Y$.

We make the following definition, analogous to Definition \ref{defn: K3 associated}:

\begin{definition}\label{defn: cubic associated}
        Let $\sigma\in \calP_{d}\setminus \calP_{22}$ for some $d$, and let $N_d\subset \NS(\DV)$ be the corresponding marking of $\DV.$  Then a marked cubic fourfold $(Y,K_d)$ is \textbf{associated to $\sigma$} if
        there exists a Hodge isometry
        \[
        H^2(\DV,\bZ)\supset N_d^\perp\overset{\sim}\longrightarrow K_d^\perp(-1)\subset H^4(Y,\bZ)(-1).
        \]
\end{definition}
    
As with associated K3 surfaces, the following lemma shows that when $\pesk$ is smooth and of the expected dimension, the existence of a cubic fourfold associated to $\sigma$ is equivalent to the natural notion of the existence of a cubic fourfold associated to $\pesk$. 
The proof is nearly identical to that of Lemma \ref{lem: equiv defn assoc k3}.

\begin{lemma}
    Let $\sigma\in \calP_d\setminus (\calP_{22}\cup \calP_{24})$. Then there is a marked cubic fourfold $(Y,K_d)$ associated to $\sigma$ if and only if
     there is a marking $M_d\subset H^{3,3}(\pesk,\bC)\cap H^6(\pesk,\bZ)$ and a Hodge isometry
    \[
    H^6(\pesk,\bZ)\supset M_d^\perp\overset{\sim}\longrightarrow K_d^\perp\subset H^4(Y,\bZ).
    \]
\end{lemma}

We classify which special Peskine sixfolds arise from trivectors  with associated cubic fourfolds:

\begin{proposition}\label{prop:associated cubic}
    Suppose $\pesk$ is a special Peskine sixfold of discriminant $d$. 
    Then $\sigma$ has an associated cubic fourfold $(Y,K_d)$ of discriminant $d$ if and only if $d$ satisfies the following conditions:
    \begin{itemize}
        \item[(a)] $d\equiv0$ or $2\bmod6$,
        \item[(b)] $d$ is not divisible by $9$ or $121$,
        \item[(c)] $33$ is a square modulo each prime divisor of $d$, and
        \item[(d)] if $d\equiv0\bmod66$, then $d$ has an odd number of prime divisors equivalent to $2$ modulo $3$, counting multiplicity.
    \end{itemize}
\end{proposition}
\begin{proof}
    We argue as in Proposition~\ref{prop:associated k3}. 
    When $\sigma \not\in \calP_{22}\cup \calP_{24},$ by the Torelli theorem, $\sigma$ has an associated cubic fourfold $Y$ if and only if $X_1^\sigma$ has a marking $M_d$ of discriminant $d$ such that 
    \[
    M_d^\perp\cong K_d^\perp\subset I_{21,2}
    \]
    for some saturated lattice $K_d$ of rank 2 containing a class $h$ of square $3$. 
    Automatically, this means $d=\mathrm{disc}(K_d)$. 
    In the case of $d=24$, we have $N_d^\perp\cong K_d^\perp(-1)$, and the conclusion stays the same.
    Cubic fourfolds can only have markings of discriminant $d\equiv0$ or $2\bmod 6$ \cite[Theorem~1.0.1]{Hassettcubicslong}, and we consider each case separately.

    \begin{itemize}
        \item[\bf Case 1:] $d\equiv2\bmod6$, and $d\not\equiv0\bmod22$. 
        Then by Lemma~\ref{lemma:disc form} and Lemma~\ref{lem: discr2}, $D(M_d^\perp)$ is cyclic, with a generator of square $-\frac{11}d$; similarly, $D(K_d^\perp)$ is cyclic with a generator of square $\frac{2d-1}{3d}$ \cite[Proposition 3.2.5]{Hassettcubicslong}.
        An isomorphism of quadratic groups between $D(M_d^\perp)$ and $D(K_d^\perp)$ sends a generator to a generator, and any generator of $D(M_d^\perp)$ has square $-\frac{11k^2}{d}$ for some integer $k$.
        So, such an isomorphism exists when there is a solution $k\in\bZ$ to the equation
        \[
        -\frac{11k^2}d=\frac{2d-1}{3d}\in\bQ/2\bZ.
        \]
        Clearing denominators, this is equivalent to requiring that
        \[
        -33k^2\equiv 2d-1 \bmod6d.
        \]
        Write $d=2^\ell p_1^{e_1}\dots p_m^{e_m}$ for distinct primes $p_i\neq2,3,11$. 
        To solve the equation above for $k$ is equivalent to solving the equation modulo $2^{\ell+1}$ and $p_i$ for all $i$. 
        If $\ell=1$, then the equation is satisfied modulo $4$; if $\ell\ge2$, then the equation reduces to $k^2\equiv1\bmod8$, and a solution can be lifted to $\bZ/2^{\ell+1}\bZ$ via Hensel's lemma, which says that a simple root $r$ of a polynomial modulo a prime power $p^s$ can be lifted to a new root $r'$ modulo $p^{s+1}$ such that $r'\equiv r\bmod p^s$.
        
        Reducing modulo $p_i$, we have $33k^2\equiv1\bmod p_i$, which has a solution if and only if $33$ is a square modulo $p_i$.
        
        \item[\bf Case 2:] $d\equiv0\bmod6$, and $d\not\equiv0\bmod22$. 
        Again, $D(M^\perp_d)$ or $D(N^\perp_{24})$ when $d=24$ is cyclic, with a generator of square $-\frac{11}d$. 
        For $D(K^\perp_d)$  to be cyclic, by \cite[Proposition~3.2.5]{Hassettcubicslong} we must have $9\nmid d$; in that case, $D(K^\perp_d)$ has a generator of square $\frac23-\frac3d$. 
        So, we are looking for a solution $k$ to 
        \[
        -\frac{11k^2}{d}=\frac23-\frac3d\in\bQ/2\bZ,
        \]
        i.e.
        \[
        -11k^2\equiv 2e-3\bmod2d,
        \]
        where $e=d/3$.
        As before, write $d=2^\ell 3p_1^{e_1}\dots p_m^{e_m}$ for distinct primes $p_i\neq2,3,11$. 
        A solution exists modulo $4$ (if $\ell=1$) or can be lifted from a solution modulo $8$ (if $\ell\ge2$), so the equation has a solution if and only if it has a solution modulo $3$ and $p_i$ for each $i$. 

        Reducing modulo $p_i$, we have $11k^2\equiv 3\bmod p_i$, so the equation has a solution modulo $p_i$ if and only if
        \[
        \left(\dfrac{11}{p_i}\right)=\left(\dfrac{3}{p_i}\right),
        \]
        i.e. 
        \[
        \left(\dfrac{33}{p_i}\right)=1,
        \]
        meaning $33$ is a square modulo $p_i$.

        Finally, we show that if there is a solution modulo $p_i$ for each $i$, then there is automatically a solution modulo $3$. 
        Using the above, along with quadratic reciprocity, we have
        \[
        \left(\dfrac{p_i}{11}\right)=\left(\dfrac{p_i}{3}\right).
        \]
        Since $3$ is a square modulo $11$ but $d$ is not (see Proposition~\ref{prop:d mod 22}), we get that $e=d/3$ is not a square modulo $11$, and it follows that
        \begin{align*}
        -1&=\left(\dfrac{e}{11}\right)=\left(\dfrac{2}{11}\right)^\ell\left(\dfrac{p_1}{11}\right)^{e_1}\dots\left(\dfrac{p_m}{11}\right)^{e_m}=\left(\dfrac{2}{3}\right)^\ell\left(\dfrac{p_1}{3}\right)^{e_1}\dots\left(\dfrac{p_m}{3}\right)^{e_m}=\left(\dfrac{e}{3}\right).
        \end{align*}
        Hence $2e$ is a square modulo $3$, and the reduction of our equation modulo $3$, i.e. $k^2=2e\bmod3$, has a solution.

        \item[\bf Case 3:] $d\equiv2\bmod6$ and $d\equiv0\bmod22$. 
        As in case $1$, $D(K_d^\perp)$ is cyclic, with a generator of square $\frac{2d-1}{3d}$. 
        For $D(M_d^\perp)$ to be cyclic, we must have $11^2\nmid d$, and in that case, by Lemma~\ref{lem: disc1} there is a generator of square $\frac8{11}-\frac{11}{d}$. 
        An isomorphism of quadratic forms requires an integer $k$ such that
        \[
        \frac{k^2(2d-1)}{3d}=\frac8{11}-\frac{11}d\in\bQ/2\bZ.
        \]
        Writing $e=\frac{2d-1}3$ and $d=11d'=11\cdot2^\ell p_1^{e_1}\dots p_m^{e_m}$ for distinct primes $p_i\neq2,3,11$, this is equivalent to requiring that
        \[
        k^2e\equiv8d'-11\bmod2d.
        \]
        The equation above has a solution if and only if 
        \[
        k^2(2d-1)=24d'-33
        \]
        has a solution modulo $2^{\ell+1}$, $11$, and each $p_i$.

        Another application of Hensel's lemma implies that $33$ is a square modulo each power of $2$, so the equation has a solution modulo $2^{\ell+1}$.
        Reducing the equation modulo $p_i$ yields $k^2\equiv33\bmod p_i$, so we now assume $33$ is a square modulo each $p_i$. Then 
        \[
        1=\left(\dfrac{33}{p_i}\right)=\left(\dfrac{3}{p_i}\right)\left(\dfrac{11}{p_i}\right)=\left(\dfrac{p_i}{3}\right)\left(\dfrac{p_i}{11}\right),
        \]
        so
        \begin{align*}
        \left(\dfrac{d'}{11}\right)&=\left(\dfrac{2}{11}\right)^\ell\prod_{i=1}^m\left(\dfrac{p_i}{11}\right)^{e_i}\\
        &=(-1)^\ell\prod_{i=1}^m\left(\dfrac{p_i}{3}\right)^{e_i}\\
        &=(-1)^\ell\prod_{p_i\equiv2\bmod3}(-1)^{e_i}.
        \end{align*}
        Modulo $11$, the equation reads $k^2\equiv9d'\bmod11$. 
        By the above, for this equation (and the others) to have a solution  means that the number of prime divisors of $d'$ equivalent to $2$ modulo $3$ (counting multiplicity, and including $2$) is even. 
        But this condition is satisfied automatically since $11d'=d\equiv2\bmod3$.
        
        \item[\bf Case 4:] $d\equiv0\bmod66$. 
        Then $D(M_d^\perp)\cong\bZ/11\bZ\times\bZ/6d'\bZ$, and $D(K_d^\perp)\cong\bZ/3\bZ\times\bZ/22d'\bZ$ where $d=66d'$. 
        For these groups to be isomorphic, they must both be cyclic, which is the case when $3,11\nmid d'$. 
        In that case, finding an isomorphism of discriminant groups compatible with the discriminant form is equivalent to solving the equation 
        \[
        k^2\left(\frac23-\frac3d\right)=\frac8{11}-\frac{11}d\in\bQ/2\bZ,
        \]
        or equivalently, 
        \[
        k^2(44d'-3)\equiv48d'-11\bmod2d.
        \]
        Write $d'=2^\ell p_1^{e_1}\dots p_m^{e_m}$ for distinct primes $p_i\neq2,3,11$. 
        We need to check when  there is a solution modulo $2^{\ell+2}$, $3$, $11$, and $p_i$ for each $i$.

        Taking the equation modulo $p_i$, we have $3k^2\equiv11\bmod p_i$, which has a solution if and only if 
        \[
        \left(\dfrac{3}{p_i}\right)=\left(\dfrac{11}{p_i}\right),
        \]
        in which case $33$ is a square modulo $p_i$. 
        By quadratic reciprocity, this is also equivalent to
        \[
        \left(\dfrac{p_i}{3}\right)=\left(\dfrac{p_i}{11}\right).
        \]
        
        Reducing modulo $11$, we obtain $3k^2=7d'\bmod11$, which has a solution if and only if $d'$ is not a square modulo $11$. 
        Assuming the equation above for each $p_i$, this means
        \begin{align*}
        -1&=\left(\dfrac{d'}{11}\right)=\left(\dfrac{2}{11}\right)^\ell\left(\dfrac{p_1}{11}\right)^{e_1}\dots\left(\dfrac{p_m}{11}\right)^{e_m}\\
        &=(-1)^\ell\left(\dfrac{p_1}{3}\right)^{e_1}\dots\left(\dfrac{p_m}{3}\right)^{e_m}\\
        &=(-1)^\ell\prod_{p_i\equiv2\bmod3}(-1)^{e_i},
        \end{align*}
        i.e. the number of prime divisors of $d'$ (hence also of $d=66d'$) equivalent to $2$ modulo $3$ (counting multiplicity and including $2$) is odd.

        Reducing modulo $3$, the equation becomes $2k^2d'\equiv1\bmod3$.
        This equation also has a solution if and only if the number of prime divisors of $d'$ equivalent to $2$ modulo $3$ (counting multiplicity and including $2$) is odd, which we have seen is a consequence of the equation having a solution modulo $11$ and each prime $p_i$.

        Finally, we show that the equation has a solution modulo $2^{\ell+2}$. 
        Reducing modulo 4 yields $3k^2\equiv3\bmod4$, which has a solution; again, we can lift to a solution of our equation modulo $2^{\ell+2}$ using Hensel's lemma.\qed
        \end{itemize}
        \hideqed
\end{proof}

For example, for Peskine sixfolds of discriminants $d=24$, $32$, $44$, $62$, $68$, $74$, and $96$, $\sigma$ has an associated cubic fourfold. 

\begin{remark}\label{remark:birfano}
    One way that a cubic fourfold $Y$ can be associated to a trivector $\sigma$ is if the Fano variety of lines $F(Y)$ is birational to $\DV$, which we explain here. 
    A birational map $F(Y)\dashrightarrow \DV$ induces a Hodge isometry $H^2(F(Y),\bZ)\cong H^2(\DV,\bZ)$, restricting to a Hodge isometry on transcendental cohomology. 
    This yields isomorphisms of Hodge structures:
    \[
    T(Y)\cong T(F(Y))(-1)\cong T(\DV)(-1)\cong T(\pesk).
    \]
    If $\sigma\in \calP_d$ is very general then the algebraic cohomology for $Y$ (resp. $\pesk$) will be isomorphic to a marking of discriminant $d$. It follows that $Y$ is associated to the $\sigma$ producing $\pesk$.
    However, Table 1 shows that a birational equivalence between $X_6^\sigma$ and $F(Y)$ is not necessary for $Y$ to be associated to $\sigma$. 
    It seems likely that geometric explanations for cubic fourfolds associated to Peskine sixfolds vary widely across different discriminants, as appears to be the case with K3 surfaces associated to cubic fourfolds.
\end{remark}

For the general Peskine sixfold $\pesk$ of discriminant $d$ with $d\leq 50$, we record whether $\sigma$ has an associated K3 surface and/or cubic fourfold in Table \ref{tab:discriminants}, as well as the related conditions of whether the associated Debarre--Voisin variety $\DV$ is birational to the Hilbert square of a K3 surface and/or to the Fano variety of lines on a cubic fourfold. We prove the claims on birationalities below:

\begin{proposition}\label{prop:low discr}
    Let $X_6^{\sigma}$ be a very general Debarre--Voisin fourfold of discriminant $d\le 50$. 
    \begin{enumerate}
        \item $X_6^{\sigma}$ is birational to the Hilbert square of a K3 surface if and only if $d=22,30,46.$
        \item $X_6^{\sigma}$ is birational to the Fano variety of lines on a cubic fourfold if and only if $d=24,32.$
    \end{enumerate}
\end{proposition}

\begin{proof}
\begin{enumerate}
        \item The case $d=22$ was addressed in \cite[Theorem 3.3]{DV}; see also \cite[Remark 3.7]{DHOGV}. In \cite[Theorem 3.1]{DHOGV}, the authors prove that the general Debarre--Voisin fourfold of discriminant $d$ is birationally equivalent to the Hilbert square of a K3 surface when $d=2(m^2+m+3)$; this verifies the claim for $d=30$ and $46$.

        We can also rule out any discriminants $d$ for which there is no associated K3 surface, leaving us with the case $d=50$. By \cite[Section 3.1, Equation (5)]{DHOGV}, a necessary condition for the Debarre--Voisin fourfold of discriminant $d$ to be birational to the Hilbert square of a K3 surface is the existence of an integer solution to the Pell-type equation $a^2-2db^2=-11$. There are no such solutions for $d=50$.

        \item As mentioned in Remark~\ref{remark:birfano}, if $\DV$ is birational to the Fano variety $F(Y)$ of lines on a very general cubic fourfold of discriminant $d$, then $\pesk$ is Hodge-theoretically associated to $Y$. This allows us to rule out all discriminants in question except $d\in\{24,32,44\}$. If the very general Debarre--Voisin fourfold of discriminant $d$ is birational to a Fano variety of lines on a cubic fourfold, then 
        the cubic fourfold is also very general of discriminant $d$: indeed, both hyperk\"ahler manifolds have the same period under the period map \cite{huybrechts}.
        Thus we have the converse: the Fano variety of lines of the very general cubic fourfold of discriminant $d$ is birational to a Debarre--Voisin fourfold. 
        So, for the remaining discriminants, we search in $\Mov(F(Y))$ for a class of square $22$ and divisibility $2$, where $Y$ is a very general cubic fourfold of discriminant $d$.

        Let $g$ denote the Pl\"ucker polarization on $F(Y)$. Using the description of algebraic classes on $Y$ and the Abel-Jacobi map \cite[Propositions 3.1.1, 3.2.2]{Hassettcubicslong}, there is a basis $\{g,\lambda\}$ of $\NS(F(Y)))$ under which the Beauville--Bogomolov form has the Gram matrix
        \[
        \begin{pmatrix}
            6 & 0 \\
            0 & -\frac d3
        \end{pmatrix}
        \text{ or } 
        \begin{pmatrix}
            6 & 2 \\
            2 & -\frac{d-2}3
        \end{pmatrix}
        \]
        depending on whether $d\equiv0$ or $2\bmod 6$, respectively. We identify at least one class $mg+n\lambda$ of square $22$:
        \begin{center}
        \begin{tabular}{|l|l|l|l|}
        \hline
        $d$     & $24$  & $32$  & $44$  \\ \hline
        $(m,n)$ & $(3,2)$ & $(7,-4)$ & $(3,2)$ \\ \hline
        \end{tabular}
        \end{center}
        In each case, $m$ is odd and $n$ is even. Since $\div_{H^2(F,\bZ)} g=2$, each class identified in the table above has divisibility $2$. Recall that the positive cone of $F(Y)$ admits a wall-and-chamber decomposition, where $\Mov(F(Y))$ is one such chamber (see for instance \cite[Theorem 3.16]{debarre2020hyperkahler}). The group $\Mon^2_{\mathrm{Hdg}}(F(Y))$ of Hodge structure preserving monodromy operators  acts transitively on these chambers \cite[Lemma 5.11 (5)]{markman}, so there exists some $\gamma\in \Mon^2_{\mathrm{Hdg}}(F(Y))$ such that $\gamma(mg+n\lambda)\in \Mov(F(Y))$; this class necessarily has square 22 and divisibility 2. \qed
    \end{enumerate}
    \hideqed
    \end{proof}
In fact, there exist discriminants $d$ that exhibit all possible combinations of behavior with respect to the four properties in question, studied in Example~\ref{missing d}. We exhibit the behavior of discriminants $d\le 50$ discussed in Proposition \ref{prop:low discr} along with these examples in Table \ref{tab:discriminants}.

\begin{example}\label{missing d}  For each of the following, let $\DV$ be a general Debarre Voisin fourfold of discriminant $d$.
\begin{enumerate}
    \item[$d=74$.] We see that $d=74$ has both an associated K3 surface and cubic fourfold by Propositions \ref{prop:associated k3} and \ref{prop:associated cubic}. As in Proposition \ref{prop:low discr} (1), we rule out $\DV$ being birational to the Hilbert square of a K3 surface by checking that there are no integer solutions to $a^2-2db^2=-11$ for $d=74$. As in Proposition \ref{prop:low discr} (2), the general Fano variety of lines of discriminant $d$ is birational to $\DV$ if and only if there exists a class of square 22 and divisibility 2 in $\Mov(F(Y))$. Using the same notation, the class $7g+4\lambda$ is such a class for $d=74$.

    \item[$d= 194$.] We see that $d=194$ has both an associated K3 surface and cubic fourfold by Propositions \ref{prop:associated k3} and \ref{prop:associated cubic}. Again applying the argument of Proposition \ref{prop:low discr} (2) shows that $\DV$ is birational to the Fano variety of lines on a cubic fourfold of discriminant $194$ (the class is $29g-8\lambda)$. Applying \cite[Theorem 2]{Add16} shows that this Fano variety of lines is birational to the Hilbert square of a K3 surface, since $194$ is of the form 
        \[
        d=\frac{2n^2+2n+2}{a^2}
        \]
        for $n=423$ and $a=43$.
    \item[$d=618$.] We see that $d=618$ has both an associated K3 surface and cubic fourfold by Propositions \ref{prop:associated k3} and \ref{prop:associated cubic}. By applying \cite[Theorem 3.1]{DHOGV}, we find that $\DV$ is birational to the Hilbert square of a K3 surface. We see that $\DV$ is not birational to the Fano variety of lines of a cubic fourfold $Y$ because there are no classes of square 22 in $\NS(F(Y))$ for $Y$ a general cubic fourfold of discriminant $d=618$.

    \item[$d=998$.] We see that $d=998$ has both an associated K3 surface and cubic fourfold by Propositions \ref{prop:associated k3} and \ref{prop:associated cubic}. 
    As above, there are no integer solutions to $a^2-2db^2=-11$ for $d=998$.
    Similarly, there are no classes of square 22 in $\NS(F(Y))$ for $Y$ a general cubic fourfold of discriminant $d=998$.

    \item[$d=2312$.] We see that $d=2312$ has an associated cubic fourfold but not an associated K3 surface by Propositions~\ref{prop:associated k3} and \ref{prop:associated cubic}. It follows that $\DV$ is not birational to the Hilbert square of any K3 surface. As in the proof of Proposition~\ref{prop:low discr} (2), there are no square 22 classes in $\NS(F(Y))$ for $Y$ a general cubic fourfold of discriminant $d=2312$.
        
\end{enumerate}
\end{example}

\newcolumntype{Y}{>{\centering\arraybackslash}X}

\begin{table}[h!]
\centering
\caption{Discriminants $d$ and their associated geometries.}
\label{tab:discriminants}
{\renewcommand*{\arraystretch}{1.1}
\begin{tabularx}{\textwidth}{|c|Y|Y|Y|Y|}
\hline
$d$ & \textbf{Associated K3 surface $S$} & \textbf{Associated cubic fourfold $Y$} & \textbf{$\DV$ birational to ${\rm Hilb}^{[2]}(S)$} & \textbf{$\DV$ birational to $F(Y)$} \\
\hline
22 & \checkmark & & \checkmark & \\
\hline
24 & & \checkmark & & \checkmark \\
\hline
28 & & & & \\
\hline
30 & \checkmark & & \checkmark & \\
\hline
32 & & \checkmark & & \checkmark \\
\hline
40 & & & & \\
\hline
44 & & \checkmark & & \checkmark \\
\hline
46 & \checkmark & & \checkmark & \\
\hline
50 & \checkmark & & & \\
\hline
\hline
74 & \checkmark & \checkmark & & \checkmark \\
\hline
194 & \checkmark & \checkmark & \checkmark  & \checkmark \\
\hline
618 & \checkmark & \checkmark & \checkmark & \\
\hline
998 & \checkmark & \checkmark & & \\
\hline
2312 & & \checkmark & & \\
\hline
\end{tabularx}
}
\end{table}

\section{Peskine sixfolds of discriminant 24}\label{sec:discr 24}

As previously mentioned, for a general trivector $\sigma\in\calP_{24}$, the Debarre--Voisin variety $\DV$ is smooth, but the Peskine sixfold $\pesk\subset\bP(V_{10})$ has one isolated singularity. More specifically, by \cite[Proposition 4.2]{BenedettiSong}, there is a unique flag $W_1\subset W_6\subset V_{10}$ (with subscripts indicating the dimensions of the subspaces) such that $\sigma(W_1,W_6,V_{10})=0$, and $\pesk$ is singular at $[W_1]$. From Proposition~ \ref{prop:associated cubic}, we know $\sigma$ has an associated cubic fourfold $Y$, which we identify explicitly in this section. 
We show, \emph{a fortiori}, that the Fano variety of lines $F(Y)$ on $Y$ is isomorphic to $\DV$ (see Remark~\ref{remark:birfano}).

\subsection{Lines on Peskine sixfolds}
First, we recall and extend some constructions from \cite[Section 2.7]{BenedettiSong} related to the Fano variety of lines $F(\pesk)$ on $\pesk$ and to the sixfold
\[
X^\sigma_7=\{[V_7]\;|\;\rank(\sigma|_{V_7})\le5\}\subset\Gr(7,V_{10}),
\]
where $\rank(\sigma|_{V_7}) = 7 - \dim \ker(\sigma|_{V_7})$ and $\ker (\sigma|_{V_7})=\{v \in V_{7} : \sigma(v, V_7, V_7)=0\}$.
For general $\sigma\in\calM_\sigma\setminus(\calP_{22}\cup\calP_{24}\cup\calP_{28})$, \cite[Theorem 2.20]{BenedettiSong} describes an isomorphism $X_7^\sigma \simeq F(\pesk)$\footnote{While the statement of \cite[Theorem 2.20]{BenedettiSong} is correct, the proof is incomplete. We follow the outline of their argument but present more details along with some new arguments to extend the results to $\mathcal{P}_{24}$. The authors have informed us they will provide a different proof in an upcoming version.}. We obtain a similar result in discriminant $24$:

\begin{proposition}\label{prop:little phi}
    Suppose $\sigma\in\calP_{24}\setminus(\calP_{22}\cup\calP_{28})$ and that $\pesk$ has finitely many singular points, both of which conditions are general in $\calP_{24}$. Then there is a morphism $\varphi:X_7^\sigma\to F(\pesk)$ which is a bijection on closed points. 
\end{proposition}

\begin{remark}\label{remark: P_28}
    The divisor $\calP_{28}$ was studied in \cite{BenedettiSong}. For general $\sigma\in \calP_{28}$, there is a distinguished flag $V_4\subset V_7\subset V_{10}$ such that $\sigma(V_4,V_7,V_7)=0.$ The associated Peskine sixfold is smooth, but the Debarre--Voisin variety $\DV$ contains a Lagrangian plane.
\end{remark}

We start by producing an injective morphism from $X_7^\sigma$ to  $F(\pesk)$.
\begin{lemma}\label{lem:phiinjective}
Suppose $\sigma\in\calP_{24}\setminus(\calP_{22}\cup\calP_{28})$ and that $\pesk$ has finitely many singular points.
    There is an injective morphism $\varphi:X_7^\sigma\to F(\pesk).$
\end{lemma}
\begin{proof}
    As done after Remark 2.19 in \cite{BenedettiSong}, we define a map $\varphi:X_7^\sigma\to F(\pesk)$ by the formula $\varphi([V_7])=[\bP(\ker(\sigma|_{V_7}))]$. Let us recall why this is well-defined. Since $\sigma\not\in\calP_{28}$, we have $\rank (\sigma|_{V_7})=5$, so $V_2=\ker(\sigma|_{V_7})$ is $2$-dimensional. Since $\sigma(V_2,V_7,V_7)=0$,  for each $1$-dimensional subspace $V_1\subset V_2$, we have $\rank(\sigma(V_1,-,-))\le6$, so $[V_1]\in \pesk$. In particular, $[\bP (V_2)]\in F(\pesk)$. 

    Next we claim that $\varphi$ is injective: suppose that there exist $V_7\neq V_7'$ such that $V_2:=\ker(\sigma|_{V_7})=\ker(\sigma|_{V_7'}).$ Let $W=V_7+V_{7}'\subset V_{10}$, so $\dim W\geq 8$, and $V_2\subset \ker\sigma|_W.$ 
    For any $V_1\subset V_2$, we have $\sigma(V_1,W,W)=0$, so $\rank (\sigma (V_1, -, -))\leq 4$. Applying \cite[Proposition 2.9]{BenedettiSong} to the flag $V_1\subset U$ for any $6$-dimensional subspace $U$ with $V_1\subset U\subset\ker(\sigma(V_1,-,-))$ shows that $\pesk$ is singular at $[V_1]$. Varying the $V_1\subset V_2$ shows $\pesk$ is singular along all of $\bP (V_2)$,  violating our generality assumption on the number of singular points. 
\end{proof}
In order to prove Proposition \ref{prop:little phi} in its entirety, we need the following technical lemma. 

\begin{lemma}\label{lemma:sum of kernels}
    Suppose $\sigma\not\in\calP_{22}\cup\calP_{28}$ and that $\pesk$ has only finitely many singular points. If $\bP(V_2)\subset\pesk$ is a line not passing through any of the singular points of $\pesk$, then the sum
    \[
    U=\sum_{v\in V_2}\ker(\sigma(v,-,-))
    \]
    is $7$-dimensional.
\end{lemma}
\begin{proof}
    Since $\bP(V_2)\subset\pesk$, each summand in the definition of $U$ is at least $4$-dimensional, and since $\bP(V_2)$ does not pass through any singular points of $\pesk$, each summand is exactly $4$-dimensional. Thus $4\le\dim (U)\le8$, and we rule out all cases except $\dim(U)=7$.

    \begin{itemize}
    \item[\bf Case 1:] If $\dim(U)=4,$ then $U=\ker(\sigma(V_1,-,-)$ for all $V_1\subset V_2.$ It follows that $V_2\subset U.$ Choose a $V_3\subset U$ containing $V_2,$ and write $U=V_3\oplus\langle x\rangle$ and $V_3=V_2\oplus\langle y\rangle$. Since $\sigma(V_2, V_2\oplus \langle y\rangle, V_{10})=0$ and $\sigma( y, V_2,V_{10})=\sigma(V_2,y,V_{10})=0,$ we see that $\sigma(V_3,V_3,V_{10})=0.$ Applying Definition \ref{defn:flag divisor}, this means that $\sigma\in \calP_{22},$ violating our generality assumption.
    
    \item[\bf Case 2:] If $\dim(U)=5,$ let $V_2=\langle x_1,x_2\rangle$ and $K_i=\ker(\sigma(x_i,-,-))$. The intersection $K_1\cap K_2=V_3$ must be $3$-dimensional, and $\sigma(V_2,V_3,V_{10})=0$. We consider two subcases:
    \begin{enumerate}
        \item If $\sigma(V_2,V_2,V_{10})=0$, then $V_2\subset K_i$, so $V_2\subset V_3$, and $V_3=V_2\oplus\langle y\rangle$ for some $y$. Then from the observations $\sigma(V_2,V_3,V_{10})=0$, $\sigma(V_2,y,V_{10})=0$, $\sigma(y,V_2,V_{10})=0$, and $\sigma(y,y,V_{10})=0$, linearity implies $\sigma(V_3,V_3,V_{10})=0$. As in the first case, this means $\sigma\in\calP_{22}$.
        \item Otherwise, $\sigma(V_2,V_2,V_{10})\neq 0.$ First, we show that $V_2\cap V_3=\{0\}$. Suppose $v\in V_2\cap V_3$, and write $v=a_1x_1+a_2x_2$. Using linearity,  $0=\sigma(x_i, v, V_{10})=\sigma(x_i,a_jx_j,V_{10})$, so if $a_j\neq0$, then $x_j \in K_i$. Hence if $v\neq0$, we have the contradiction $\sigma(V_2,V_2,V_{10})= 0$.
        So, $U=V_2\oplus V_3$. Let $V_9=\ker\sigma(V_2,V_2,-)$, noting that $U\subset V_9$ and $\sigma(V_2,U,V_9)=0$; we can thus descend $\sigma(V_2,-,-)$ to a two form on $V_9/U$. 
        
        We now look for flags $U\subset V_7\subset V_9$ such that $\sigma(V_2,V_7,V_7)=0$. Since $\sigma(V_2,U,V_9)=0$, any such $V_7$ will satisfy $\sigma(V_2,U,V_7)=0$, and thus the forms will descend to the quotient. The condition $\sigma(V_2,V_7,V_7)=0$ is equivalent to describing the locus $Z\subset\Gr(2,V_9/U)$ parametrizing $2$-dimensional subspaces of $V_9/U$ on which $\sigma(V_2,-,-)$ vanishes. This locus is cut out by the vanishing of two linear forms: $\sigma(V_2,V_7/U,V_7/U)=0$ for $U\subset V_7\subset V_9$ if $\sigma(x_1,-,-)=0$ and $\sigma(x_2,-,-)=0$ as elements of $\bigwedge^2(V_7/U)^*$, both of which conditions are determined by $2\times2$ Pfaffians. Hence $\dim Z=2$, and every $V_7$ such that $V_7/U \in Z$ satisfies $\varphi(V_7)=[\bP(V_2)]$,
        contradicting the injectivity of $\varphi$ from Lemma~\ref{lem:phiinjective}.
    \end{enumerate}
    
    \item[\bf Case 3:] Suppose $\dim(U)=6$. For each $v\in V_2$, we have a two form $\sigma(v,-,-)\in \wedge^2V_{10}^\vee$ which we can restrict to get a two form $\sigma_v:=\sigma(v,-,-)|_U\in \wedge^2 U^\vee.$ 
    We first assume for the sake of contradiction that for all nonzero $v\in V_2,$ the restricted two form $\sigma_v$ is nonzero. Thus we have a pencil of two forms $P\cong \bP(V_2)\subset \bP(\wedge^2U^\vee).$

Since $K_v:=\ker \sigma(v,-,-)\subset U$ and $\dim K_v=4,$ each two form in the pencil has $\rank(\sigma_v)=2.$
Now, a rank 2 skew-symmetric form is necessarily the wedge of two linear functionals, i.e. $\sigma_v=\alpha_v\wedge \beta_v$ for $\alpha_v, \beta_v\in U^\vee.$ Since every $\sigma_v\in P$ is decomposable, this line $P$ is contained in the image of $\Gr(2,U^\vee)$ under the Pl\"ucker embedding. There are two types of lines: (1) lines parametrizing 2-planes in $U^\vee$ containing a fixed vector, or (2) lines parametrizing 2-planes in $U^\vee$ contained in a fixed 3-plane.

\begin{itemize}
    \item[(1)] For a line of this type, for each $\sigma_v\in P$ there exists a fixed $\alpha\in U^\vee$ such that $\sigma_v=\alpha\wedge \beta_v.$ Notice that $\ker(\sigma_v)\subset \ker(\alpha).$ But $U=\sum_{v\in V_2} \ker \sigma_v \subset \ker\alpha,$ and $\dim \ker\alpha =5$, a contradiction.
    \item[(2)] In this case there exists $W_3\subset U^\vee $ so that $P\subset \bP(\wedge^2W_3)\subset \bP(\wedge^2U^\vee).$ Let 
    $$K:=\{u\in U\mid \gamma(u)=0 \text{ for all } \gamma \in W_3\},$$ and note that $\dim K= \dim U- \dim W_3=3$.
    Also notice that $K\subset \ker \sigma_v$ for all two forms in the pencil. But now for any two (linearly independent vectors) $x_1,x_2\in V_2$, $K\subset \ker \sigma(x_1,-,-)\cap \ker \sigma(x_2,-,-)$, and so $U=\ker \sigma(x_1,-,-) + \ker \sigma(x_2,-,-)$ has dimension at most 5, a contradiction.
\end{itemize}

We next consider the case where there is some nonzero $x_1\in V_2$ with $\sigma_{x_1}=0$ but not all $\sigma_v= 0$, so there is some nonzero $x_2 \in V_2$ with $\sigma_{x_2}$ nonzero. Writing $V_2=\langle x_1, x_2\rangle$, any two form in the pencil can be written as $\sigma(ax_1+bx_2,-,-)= a\sigma(x_1,-,-)+b\sigma(x_2,-,-)\in \wedge^2V_{10}^\vee$. For any $b\neq 0,$ we have $\ker \sigma_v = \ker \sigma_{x_2}$, and $K_v \subseteq \ker \sigma_v$, so that $\dim \ker \sigma_v \geq 4$. This forces $\dim \ker \sigma_v = 4$, since $\sigma_v\neq 0$. Running over all $v\in V_2$, it follows that $U\subseteq  \ker \sigma_{x_2}$, a contradiction.

Finally, suppose $\sigma_v=0$ for all $v\in V_2$. Again, we work in a basis $V_2=\langle x_1,x_2\rangle$ and set $K_i=\ker(\sigma(x_i,-,-))$. Further, set $K_i'=(K_1\cap K_2)^\perp\cap K_i$, noting that $\dim(K_i')=2$ and that $U=(K_1\cap K_2)\oplus K_1'\oplus K_2'$. Extending bases for $K_1\cap K_2$, $K_1'$, and $K_2'$ to a basis for $V_{10}$, we can write $\sigma(ax_1+bx_2,-,-)$ as a block matrix
\[
\begin{pmatrix}
    0 & 0 & 0 & 0 \\
    0 & 0 & 0 & bM_1\\
    0 & 0 & 0 & aM_2\\
    0 & -bM_1^t & -aM_2^t & N
\end{pmatrix}
\]
where $M_i\in K_i^\vee\wedge (U^\perp)^\vee$ and $N\in\wedge^2(U^\perp)^\vee$. For the rank of this matrix to be at most $6$ for all $ax_1+bx_2\in V_2$ (since $\bP(V_2)\subset\pesk$), we must have $\rank(bM_1^t \;\; aM_2^t)\le3$. If $\rank(bM_1^t \;\; aM_2^t)=3$ generically, then in a suitable basis of $(U^\perp)^\vee$, $N=0$. Evaluating at $a=0$ yields a matrix of rank at most $4$, contradicting the assumption that $\bP(V_2)$ avoids the singular locus of $\pesk$. On the other hand, if $\rank(bM_1^t \;\; aM_2^t)=2$ generically, then after changing to a suitable basis of $U^\perp$, we can write the matrix as
\[
\begin{pmatrix}
    0 & 0 & 0 & 0 & 0 & 0 & 0 & 0 & 0 & 0\\
    0 & 0 & 0 & 0 & 0 & 0 & 0 & 0 & 0 & 0\\
    0 & 0 & 0 & 0 & 0 & 0 & b & 0 & 0 & 0\\
    0 & 0 & 0 & 0 & 0 & 0 & 0 & b & 0 & 0\\
    0 & 0 & 0 & 0 & 0 & 0 & a & 0 & 0 & 0\\
    0 & 0 & 0 & 0 & 0 & 0 & 0 & a & 0 & 0\\
    0 & 0 & -b & 0 & -a & 0 & 0 & 0 & 0 & 0\\
    0 & 0 & 0 & -b & 0 & -a & 0 & 0 & 0 & 0\\
    0 & 0 & 0 & 0 & 0 & 0 & 0 & 0 & 0 & \ell\\
    0 & 0 & 0 & 0 & 0 & 0 & 0 & 0 & -\ell & 0\\
\end{pmatrix}
\]
where $\ell$ is some linear form in $a$ and $b$. Evaluating at a root of $\ell$, one again finds a vector $ax_1+bx_2$ for which the matrix has rank $4$. Finally, $\rank(bM_1^t \;\; aM_2^t)<2$, then one can locate some $u\in U^\perp$ such that either $\sigma(x_1,u,V_{10})=0$ or $\sigma(x_2,u,V_{10})=0$, contradicting the definition of $U$.
    
    \item[\bf Case 4:] If $\dim(U)=8$, let $V_2=\langle x_1,x_2\rangle$ and set $K_i=\ker(x_i,-,-)$ so that $U=K_{1}\oplus K_{2}$. 
    Let $v=ax_1+bx_2\in V_2$ and write $\sigma(v,-,-)=a\sigma(x_1,-,-)+b\sigma(x_2,-,-).$ Restricting to $U$ and working in a basis, we have that
    $$\sigma(v,-,-)|_U=a\begin{pmatrix}
        0 & 0\\
        0 & A
    \end{pmatrix}+b\begin{pmatrix}
        B & 0\\
        0 & 0
    \end{pmatrix},$$ 
    with $\ker A \subset K_{2}$ and $\ker B \subset K_{1}$. It follows that $\ker(\sigma(v,-,-))\subset \ker A\oplus \ker B$, and by varying $v,$ we see that $U\subset \ker A\oplus \ker B$. Thus, $U = \ker A\oplus \ker B$, which forces $\ker A=K_{2}$ and $\ker B=K_{1}$.  Then $\sigma(x_1,U,U)=\sigma(x_2,U,U)=0$, and so $\sigma(v,-,-)|_U=0$. Thus, for any $V_1\subset V_2$, $\rank(\sigma(V_1,-,-))\leq 4$, and $\pesk$ is singular along the line $\bP(V_2)$.
\end{itemize}

Having ruled out all other possibilities, we conclude that $\dim(U)=7$.
\end{proof}
Next, we show that a line in $\pesk$ not passing through any of the singular points is contained in the image of $\varphi: X_7^\sigma\rightarrow F(\pesk).$

\begin{lemma}\label{lemma:varphi preimages}
    Suppose $\sigma\not\in\calP_{22}\cup\calP_{28}$ and that $\pesk$ has only finitely many singular points. If $\bP(V_2)\subset\pesk$ is a line not passing through any of the singular points of $\pesk$, then there is some $7$-dimensional space $V_7\supset V_2$ such that $\sigma(V_2,V_7,V_7)=0$.
\end{lemma}
\begin{proof}
    We construct $V_7$ by working in a basis. Write $V_2=\langle x_1,x_2\rangle$. By assumption, $x_1$ and $x_2$ are smooth points of $\pesk$, so the dimension of $K_i=\ker(\sigma(x_i,-,-))$ is $4$. The subspace $U=K_1+K_2$ is $7$-dimensional by Lemma~\ref{lemma:sum of kernels}, so $K_1\cap K_2$ is $1$-dimensional. It is easy to see that $K_1\cap K_2\not\subset V_2$, or else we would have $K_1\cap K_2=V_2$. So, we can write $K_1=\langle x_1,k,y_1,y_2\rangle$, $K_2=\langle x_2,k,y_3,y_4\rangle$. We analyze three cases:
    
    \begin{itemize}
        \item[\bf Case 1:] Suppose $\sigma(x_1,y_3,y_4)=0$ and $\sigma(x_2,y_1,y_2)=0$. By linearity, $\sigma(V_2,U,U)=0$, so we can set $U=V_7$.
        \item[\bf Case 2:] Next, suppose $\sigma(x_1,y_3,y_4)=0$ but $\sigma(x_2,y_1,y_2)\neq0$. After rescaling we can assume $\sigma(x_2,y_1,y_2)=1$. 
        Choose a splitting $V_{10}=U\oplus U^\perp$. 
        Repeated use of the linearity of $\sigma$ shows that $\sigma(V_2,V_2,U)=0$. 
        Notice that if $\sigma(x_1,x_2,U^\perp)=0$, then $\sigma(V_2\oplus\langle k\rangle,V_2\oplus\langle k\rangle,U\oplus U^\perp)=0$, which would mean $\sigma\in\calP_{22}$, violating our assumptions on $\sigma$.
        Thus, to choose a basis for $U^\perp$, we can first choose some nonzero vector
        \[
        z_1\in \ker(\sigma(x_1,x_2,-))\cap\ker(\sigma(x_1,y_3,-))\cap U^\perp,
        \] 
        and then choose the remaining basis vectors $U^\perp=\langle z_1,z_2,z_3\rangle$ so that $\sigma(x_1,x_2,z_2)=1$.
        Relative to the basis $\{ x_1,x_2,k,y_1,\dots,y_4,z_1,z_2,z_3\}$ for $V_{10}$, we can represent the form $\sigma(ax_1+bx_2,-,-)$  by the skew-symmetric matrix
        \[
        \begin{pmatrix}
            0 & 0 & 0 & 0 & 0 & 0 & 0 & 0 & -b & *\\
            0 & 0 & 0 & 0 & 0 & 0 & 0 & 0 & a & *\\
            0 & 0 & 0 & 0 & 0 & 0 & 0 & 0 & 0 & 0\\
            0 & 0 & 0 & 0 & b & 0 & 0 & d_{11}b & d_{12}b & d_{13}b\\
            0 & 0 & 0 & -b & 0 & 0 & 0 & d_{21}b & d_{22}b & d_{23}b\\
            0 & 0 & 0 & 0 & 0 & 0 & 0 & 0 & c_{32}a & c_{33}a\\
            0 & 0 & 0 & 0 & 0 & 0 & 0 & c_{41}a & c_{42}a & c_{43}a\\
            0 & 0 & 0 & -d_{11}b & -d_{21}b & 0 & -c_{41}a & 0 & * & *\\
            -b & a & 0 & -d_{12}b & -d_{22}b & -c_{32}a & -c_{42}a & * & 0 & *\\
            * & * & 0 & -d_{13}b & -d_{23}b & -c_{33}a & -c_{43}a & * & * & 0
        \end{pmatrix}
        \]
        where $c_{ij}=\sigma(x_1,y_i,z_j)$ and $d_{ij}=\sigma(x_2,y_i,z_j)$. Subtracting appropriate multiples of $y_1$ and $y_2$ from the $z_j$s, and subtracting multiples of $x_2$ from $y_3$ and $y_4$, we obtain a new basis $\{x_1,x_2,k,y_1,y_2,y_3',y_4',z_1',z_2',z_3'\}$ in which the matrix becomes the following.
        \[
        \begin{pmatrix}
            0 & 0 & 0 & 0 & 0 & 0 & 0 & 0 & -b & *\\
            0 & 0 & 0 & 0 & 0 & 0 & 0 & 0 & a & *\\
            0 & 0 & 0 & 0 & 0 & 0 & 0 & 0 & 0 & 0\\
            0 & 0 & 0 & 0 & b & 0 & 0 & 0 & 0 & 0\\
            0 & 0 & 0 & -b & 0 & 0 & 0 & 0 & 0 & 0\\
            0 & 0 & 0 & 0 & 0 & 0 & 0 & 0 & 0 & c_{33}a\\
            0 & 0 & 0 & 0 & 0 & 0 & 0 & c_{41}a & 0 & c_{43}a\\
            0 & 0 & 0 & 0 & 0 & 0 & -c_{41}a & 0 & * & *\\
            -b & a & 0 & 0 & 0 & 0 & 0 & * & 0 & *\\
            * & * & 0 & 0 & 0 & -c_{33}a & -c_{43}a & * & * & 0
        \end{pmatrix}
        \]
        If $c_{33}=0$, then $\sigma(V_2,y_3,V_{10})=0$, so $y_3\in K_1\cap K_2$, a contradiction since $k$ spans $K_1\cap K_2$. Hence after rescaling, we can assume $c_{33}=1$. 
        Since $\bP(V_2)\subset\pesk$, the matrix above has rank 6 for all $(a,b)\in\bP^1$, forcing $c_{41}=0$. Setting $V_7=\langle x_1,x_2,k,y_1,y_3',y_4',z_1'\rangle$, we have $\sigma(V_2,V_7,V_7)=0$. 
        
        \item[\bf Case 3:] The last case is where (after rescaling $x_1,x_2$) we have $\sigma(x_1,y_3,y_4)=\sigma(x_2,y_1,y_2)=1$. Choose a splitting $V_{10}=U\oplus U^\perp$. As in the previous case, if $\sigma(x_1,x_2,U^\perp)=0$, then we reach the contradiction $\sigma\in\calP_{22}$. So, we can choose a basis $U^\perp=\langle z_1,z_2,z_3\rangle$ so that $\sigma(x_1,x_2,z_1)=1$, and we can choose the remaining basis vectors so that $z_2,z_3\in\ker(\sigma(x_1,x_2,-))$.
        
        Relative to the basis $\{x_1,x_2,k,y_1,\dots,y_4,z_1,z_2,z_3\}$ for $V_{10}$, we can represent the form $\sigma(ax_1+bx_2,-,-)$ by the skew-symmetric matrix
        \[
        \begin{pmatrix}
            0 & 0 & 0 & 0 & 0 & 0 & 0 & -b & 0 & 0\\
            0 & 0 & 0 & 0 & 0 & 0 & 0 & a & 0 & 0\\
            0 & 0 & 0 & 0 & 0 & 0 & 0 & 0 & 0 & 0\\
            0 & 0 & 0 & 0 & b & 0 & 0 & d_{11}b & d_{12}b & d_{13}b\\
            0 & 0 & 0 & -b & 0 & 0 & 0 & d_{21}b & d_{22}b & d_{23}b\\
            0 & 0 & 0 & 0 & 0 & 0 & a & c_{31}a & c_{32}a & c_{33}a\\
            0 & 0 & 0 & 0 & 0 & -a & 0 & c_{41}a & c_{42}a & c_{43}a\\
            b & -a & 0 & -d_{11}b & -d_{21}b & -c_{31}a & -c_{41}a & 0 & * & *\\
            0 & 0 & 0 & -d_{12}b & -d_{22}b & -c_{32}a & -c_{42}a & * & 0 & *\\
             0 & 0 & 0 & -d_{13}b & -d_{23}b & -c_{33}a & -c_{43}a & * & * & 0\\
        \end{pmatrix}
        \]
        where $c_{ij}=\sigma(x_1,y_i,z_j)$ and $d_{ij}=\sigma(x_2,y_i,z_j)$. Subtracting multiples of the $y_i$s from the $z_j$s, we obtain a new basis $\{x_1,x_2,k,y_1,y_2,y_3,y_4,z_1',z_2',z_3'\}$ under which the form is represented by the matrix
        \[
        \begin{pmatrix}
            0 & 0 & 0 & 0 & 0 & 0 & 0 & -b & 0 & 0\\
            0 & 0 & 0 & 0 & 0 & 0 & 0 & a & 0 & 0\\
            0 & 0 & 0 & 0 & 0 & 0 & 0 & 0 & 0 & 0\\
            0 & 0 & 0 & 0 & b & 0 & 0 & 0& 0 & 0\\
            0 & 0 & 0 & -b & 0 & 0 & 0 & 0& 0 & 0\\
            0 & 0 & 0 & 0 & 0 & 0 & a & 0 & 0 & 0\\
            0 & 0 & 0 & 0 & 0 & -a & 0 & 0 & 0 & 0\\
            b & -a & 0 & 0 & 0 & 0 & 0 & 0 & * & *\\
            0 & 0 & 0 & 0 & 0 & 0 & 0 & * & 0 & \ell(a,b)\\
             0 & 0 & 0 & 0 & 0 & 0 & 0 & * & -\ell(a,b) & 0\\
        \end{pmatrix}
        \]
        for some linear form $\ell(a,b)$. Since $\bP(V_2)\subset\pesk$, this matrix has rank $6$ for all $(a,b)\in\bP^1$. Hence the form $\ell(a,b)$ vanishes on every point of $\bP^1$ for which $a,b\neq0$, which is only possible if $\ell(a,b)=0$. 
        Setting $V_7=\langle x_1,x_2,k,y_1,y_3,z_2',z_3'\rangle$, we have $\sigma(V_2,V_7,V_7)=0$.\qed
    \end{itemize}
    \hideqed
\end{proof}

We are now prepared to complete the proof of Proposition \ref{prop:little phi} by proving $\varphi:X_7^\sigma\rightarrow F(\pesk)$ is surjective.

\begin{proof}[Proof of Proposition \ref{prop:little phi}]
    Let $[W_1],\dots,[W_k]\in\pesk$ be the singular points. By Lemma~\ref{lemma:varphi preimages}, the open set
    \[
    F(\pesk)\setminus\{\bP(V_2)\subset \pesk\;|\; W_i\subset V_2\;\text{ for some }\;1\le i\le k\}
    \]
    lies in the image of $\varphi$ since for each line $\bP(V_2)\subset\pesk$ avoiding the singular points, we constructed a $7$-dimensional space $V_7$ such that $\sigma(V_2,V_7,V_7)=0$. Since $\dim(\pesk)=6$, the locus of lines in $\pesk$ through $\bP(W_i)$ is at most $5$-dimensional. In particular, 
    \[
    \dim(F(\pesk)\setminus\mathrm{im}(\varphi))\le5<\dim X_7^\sigma=\dim(\mathrm{im}(\varphi)).
    \]
    Since $X_7^\sigma$ is complete, the image of $\varphi$ is closed, and by comparing dimensions above, we see $\mathrm{im}(\varphi)=F(\pesk)$. The fact that $\varphi$ is a bijection on closed points then follows from the injectivity of $\varphi$, proved in Lemma~\ref{lem:phiinjective}.
\end{proof}

\begin{remark}
    We do not claim that $X_7^\sigma$ and $F(\pesk)$ are isomorphic because we did not compute the singular locus of either variety. It seems plausible that $F(\pesk)$ has worse singularities than $X_7^\sigma$ along the locus of lines through each singular point $[W_i]$.
\end{remark}

\subsection{The cubic fourfold and its variety of lines}

From now on, we fix a distinguished flag $W_1\subset W_6\subset V_{10}$ such that $\sigma(W_1,W_6,W_{10})=0$, but we do not assume that this is the only such flag for $\sigma$. In \cite[Lemma 3.4]{BenedettiFaenzi}, the authors explain that the intersection $Y= \pesk\cap\bP (W_6)$ is a cubic fourfold, which they speculate is singular at its distinguished point $[W_1]$ in \cite[Remark 3.5]{BenedettiFaenzi}. However, we demonstrate that this is not the case:

\begin{proposition}\label{prop:smoothcubic}
    For a general $\sigma\in\calP_{24}$ with distinguished flag $W_1\subset W_6\subset V_{10}$, the cubic fourfold $Y=\pesk\cap\bP (W_6)$ is smooth.
\end{proposition}
\begin{proof}
    It is enough to construct one example for which $Y$ is smooth, as done in Appendix~\ref{appendix:examples}.
\end{proof}

When $Y$ is smooth, its Fano variety of lines $F(Y)$ is a smooth hyperk\"ahler fourfold. Using Proposition~\ref{prop:little phi}, we construct an isomorphism $\psi:F(Y)\to\DV$. Together with Proposition~\ref{prop:smoothcubic} and Remark~\ref{remark:birfano}, the following proposition proves Theorem~\ref{main thm 2}.

\begin{proposition}\label{prop:fano dv}
    Let $\sigma\in\calP_{24}$ be a trivector with distinguished flag $W_1\subset W_6\subset V_{10}$, and let $Y=\pesk\cap\bP(W_6)$. If $\sigma$ satisfies the generality assumptions
    \begin{itemize}
        \item $\sigma\not\in\calP_{22}\cup\calP_{28}$,
        \item $\pesk$ has finitely many singular points, and
        \item the cubic fourfold $Y$ is smooth,
    \end{itemize}
    then $F(Y)$ and $\DV$ are isomorphic.
\end{proposition}
\begin{proof}
    The first two generality assumptions allow us to apply Proposition~\ref{prop:little phi}, so for each line $\bP(V_2)\subset\pesk$, there is a unique $V_7\in X_7^\sigma$ such that $\varphi([V_7])=[\bP (V_2)]$. 
    Arguing as in the last paragraph of the proof of \cite[Theorem 2.20]{BenedettiSong}, there is some $[V_6]\in \DV$ such that $V_2\subset V_6\subset V_7$. We first verify that the $V_6$ is unique, yielding a morphism $\Psi:F(\pesk)\rightarrow \DV$.

    Indeed, suppose for the sake of contradiction that there exist two distinct flags $V_2\subset V_6,V_6'\subset V_7$ with $\sigma|_{V_6}=0$ and $\sigma|_{V_6'}=0$. Let $V_5=V_6\cap V_6'$, and write $V_6=V_5\oplus\langle x\rangle$ and $V_6'=V_5\oplus\langle y\rangle$. Further, let $K=\ker(\sigma(x,y,-)|_{V_5})$, noting $\dim(K)\ge4$. From the fact that $\sigma|_{V_6}=0$, we have $\sigma(K,V_6,V_6)=0$, and by construction of $K$, we have $\sigma(K,x,y)=\sigma(K,y,x)=0$, so linearity and the fact that $\sigma|_{V_6'}=0$ yields $\sigma(K,V_7,V_7)=0$. It follows that $\sigma\in\calP_{28}$, contradicting our first generality assumption.

    Let $\psi$ be the restriction of $\Psi$ to $F(Y)$, and consider the Stein factorization of $\psi$:
    \[
    F(Y)\overset{\alpha}\longrightarrow G\overset{\beta}\longrightarrow\DV.
    \]
    By the third generality assumption, $F(Y)$ is a projective hyperk\"ahler manifold. Because $\alpha_*\calO_{F(Y)}=\calO_G$ and $F(Y)$ is normal, $G$ is normal. Since $\alpha$ is a surjective morphism with connected fibers over a normal base $G$, it follows that $\alpha$ is either constant, an isomorphism, or a Lagrangian fibration over a plane (see \cite{ou} and \cite{HX} in the fourfold case).

    By Proposition~\ref{prop:little phi} the fibers of $\alpha$ embed set-theoretically in those of $\varphi\circ\Psi:X_7^\sigma\to\DV$. Note that $\varphi$ may fail to be an isomorphism only over points of the form $[\bP(V_2)]$ where $\bP(V_2)$ passes through a singular point of $\pesk$. Since $Y$ is smooth, the collection of such lines forms a one-dimensional family in $F(Y)$, so a general fiber of $\alpha$ embeds scheme-theoretically in a general fiber of $\varphi\circ\Psi$. By construction of $\Psi$, if $\Psi(\varphi(\bP(V_2)))=[V_6]$, then $V_6\subset V_7=\varphi(\bP(V_2))$, so the fiber of $\varphi\circ\Psi$ over $[V_6]$ embeds in $\bP(V_{10}/V_6)$. Hence a general fiber of $\alpha$ embeds in $\bP^3$.  In particular, $\alpha$ cannot be constant since $\dim(F(Y))=4$. We reach a similar contradiction assuming that $\alpha$ is a Lagrangian fibration. Indeed, if it is, then the general fiber of $\alpha$ is an abelian surface \cite{matsushita}. However, there is no embedding of an abelian surface in $\bP^3$.
    

    Having proved that $\alpha$ is an isomorphism, we conclude that $\psi$ is finite. Since $F(Y)$ and $\DV$ both have trivial canonical bundle, the morphism $\psi$ is unramified. On the other hand, $\DV$ is simply-connected, so no nontrivial \'etale covers exist. The result follows.
\end{proof}

Using the isomorphism above, one could hope to compare the twisted K3 surface $(S,\alpha)$ associated to $Y$ (described in \cite{hassett24}) to the twisted K3 surface $(S',\beta)$ associated to $\sigma$ (described in \cite{BenedettiFaenzi}). We expect the twisted K3 surfaces to be derived equivalent.
With a better understanding of the relationship between special surfaces in $Y$ and special threefolds in $\pesk$, we also expect that one could prove that $\pesk$ is rational when $\beta=0$, affirming Conjecture 4.1 of \cite{BenedettiFaenzi} by applying Hassett's constructions from \cite{hassett24}. Relatedly, it would be interesting to see how to produce the Peskine sixfold $\pesk$ with $\sigma$ associated to a very general cubic fourfold of discriminant $24$, reversing the construction studied here. 

\appendix

\section{An example}\label{appendix:examples}
The following computations, justifying Proposition~\ref{prop:smoothcubic}, were done in \textsf{Magma} \cite{Magma}. 

Fix a complex vector space $V_{10}$ with basis $\{v_1,\dots,v_{10}\}$. We denote by $v_i^*$ the vectors in the dual basis of $V_{10}^\vee$. The trivector $\sigma\in\bigwedge^3V_{10}^\vee$ defined below satisfies $\sigma(W_1,W_6,V_{10})=0$, where  $W_1=\langle v_1\rangle$ and $W_6 = \langle v_1,\dots,v_6\rangle$, so $\sigma\in\calP_{24}$. Writing $[i,j,k] = v^*_{i} \wedge v^*_{j} \wedge v^*_{k}$, we set
{\small
\begin{align*}
\sigma &=-4 [1,7,8] + 2 [1,7,9] - [1,7,10] - 2 [1,8,9] + [1,8,10] + 4 [1,9,10] + 3 [2,3,5] - [2,3,6]\\
    & \quad - [2,3,7] - 3 [2,3,8] - 4 [2,3,9] + 3 [2,3,10] - [2,4,5] + [2,4,7] - 2 [2,4,8] - 2 [2,4,9]\\
    &\quad  - [2,4,10] - 3 [2,5,6] + [2,5,7] - 3 [2,5,9] - 4 [2,5,10] + 2 [2,6,7] - 3 [2,6,9] + 4 [2,6,10] \\
    &\quad  + 4 [2,7,8] + [2,7,9] + 2 [2,7,10] - 2 [2,8,9] - [2,8,10] - 3 [2,9,10] - 4 [3,4,5] + 2 [3,4,6]\\
    &\quad - [3,4,7] - [3,4,8] + 4 [3,4,9] + 2 [3,4,10] - [3,5,6] - [3,5,8] - 3 [3,5,9] - [3,5,10] \\
    &\quad- 2 [3,6,8] - 3 [3,6,9] - 4 [3,6,10] + [3,7,8] + 2 [3,7,9] + [3,7,10] + [3,8,9] - 3 [3,8,10] \\
    &\quad + 3 [3,9,10] + 3 [4,5,6] + 2 [4,5,7] + 4 [4,5,9] + [4,5,10] + 2 [4,6,7] - 2 [4,6,8] - [4,6,9] \\
    &\quad+ 3 [4,6,10] + 2 [4,7,8] + 2 [4,7,9] - 3 [4,7,10] + [4,8,9] - 3 [4,8,10] - [4,9,10] + [5,6,7] \\
    &\quad - 4 [5,6,8] + 4 [5,6,9] + 3 [5,7,8] - 4 [5,7,9] + 3 [5,8,9] + 2 [5,9,10] - 2 [6,7,8] + 4 [6,7,9] \\
    &\quad- 2 [6,8,9] - 3 [6,8,10] - [7,8,9] + 4 [7,8,10] + 4 [7,9,10] + 3 [8,9,10].
\end{align*}
}

To find the $\sigma\in \bigwedge^3 V_{10}^\vee$ given above, we construct a random trivector $\sigma$ by assigning random coefficients to $[i,j,k]$ with $i<j<k$, with the coefficient on $[1,i,j]$ equal to $0$ for $i,j \leq 6$, so that $\rank(\sigma(W_1,-,-))\leq 4$. For any nonzero $v=\sum_{i=1}^{10} x_iv_i \in V_{10}$, we have
\[\sigma(v,v_i,v_j) = \sum_{k=1}^{10} x_k \sigma(v_k,v_i,v_j)=\sum_{k=1}^{10} x_k \sigma(v_i,v_j,v_k),\]
and $\sigma(v_i,v_j,v_k)$ is given by the coefficient on $[i,j,k]$. Let $M$ be the matrix whose $(i,j)$ entry is $\sigma(v,v_i,v_j)$, a linear equation in $x_1,...,x_{10}$. Then the Peskine sixfold $\pesk$ is cut out of $\bP(V_{10})$ by the $8\times 8$ principal Pfaffians of $M$.
Equations for the Peskine variety $X_1^{\sigma}$ can be found in the associated \textsf{Magma} code, which also verifies that the cubic fourfold $Y=X_1^{\sigma}\cap \bP (W_6)$, with defining equation below, is smooth.

\begin{align*}
f =& 16x_{1}^2x_{2} - 133x_{1}x_{2}^2 + 223x_{2}^3 - 32x_{1}^2x_{3} + 121x_{1}x_{2}x_{3} - 485x_{2}^2x_{3} + 578x_{1}x_{3}^2 - 340x_{2}x_{3}^2 \\
&- 18x_{3}^3 + 32x_{1}^2x_{4} - 646x_{1}x_{2}x_{4} + 471x_{2}^2x_{4} - 58x_{1}x_{3}x_{4} - 199x_{2}x_{3}x_{4} - 189x_{3}^2x_{4}  \\
&- 236x_{1}x_{4}^2 + 304x_{2}x_{4}^2 - 270x_{3}x_{4}^2 + 50x_{4}^3 + 320x_{1}^2x_{5} - 1028x_{1}x_{2}x_{5} + 530x_{2}^2x_{5} \\
&- 395x_{1}x_{3}x_{5} + 456x_{2}x_{3}x_{5} - 677x_{3}^2x_{5} - 782x_{1}x_{4}x_{5} + 1041x_{2}x_{4}x_{5} - 811x_{3}x_{4}x_{5}  \\ 
& - 406x_{4}^2x_{5} + 49x_{1}x_{5}^2 + 681x_{2}x_{5}^2 - 200x_{3}x_{5}^2 + 150x_{4}x_{5}^2 - 312x_{5}^3 - 144x_{1}^2x_{6}  \\
  &+ 1098x_{1}x_{2}x_{6} - 643x_{2}^2x_{6} - 40x_{1}x_{3}x_{6} - 144x_{2}x_{3}x_{6} + 1074x_{3}^2x_{6} + 413x_{1}x_{4}x_{6} \\ 
  &- 792x_{2}x_{4}x_{6} - 452x_{3}x_{4}x_{6} - 231x_{4}^2x_{6} + 19x_{1}x_{5}x_{6} + 18x_{2}x_{5}x_{6} + 502x_{3}x_{5}x_{6} \\
  &- 936x_{4}x_{5}x_{6} + 61x_{5}^2x_{6} - 567x_{1}x_{6}^2 + 1384x_{2}x_{6}^2 + 115x_{3}x_{6}^2 + 527x_{4}x_{6}^2 \\ 
  &+ 164x_{5}x_{6}^2 - 405x_{6}^3.
\end{align*}

\bibliographystyle{alpha}
\bibliography{bibliography}
\end{document}